\documentclass[11pt]{article}

\usepackage{amsmath}
\usepackage{amsfonts}
\usepackage{amsthm}
\usepackage{amsbsy}

\usepackage[english]{babel} 

\usepackage{amssymb}
\usepackage{graphicx}        
\usepackage{cite}                            
\usepackage{epstopdf}
\usepackage{epsfig}
\usepackage{subfigure}
\usepackage{xcolor}

\newtheorem{remark}{Remark}[section]
\newtheorem{lemma}{Lemma}[section]
\newtheorem{theorem}{Theorem}[section]

\newtheorem{algorithm}{Algorithm}[section]

\def\RR{\mathbb R}
\def\R{\mathbb R}
\def\E{\mathbb E}

\def\ind{\hspace{0.25in}}
\def\fP{\em \sf}
\def\nt{n_{\mbox{\sc \small tot}}}
\def\cir{\makebox[0.5cm]{$\circ$}}
\def\IR{\mathop{\mbox{\rm Sround}}\nolimits}
\newcommand{\intg}[1]{\lfloor #1\rfloor}

\def\be{\begin{equation}}
\def\ee{\end{equation}}
\def\bea{\begin{eqnarray}}
\def\eea{\end{eqnarray}}
\def\beas{\begin{eqnarray*}}
\def\eeas{\end{eqnarray*}}

\def\rand{\xi}

\title{Monte Carlo stochastic Galerkin methods for the Boltzmann equation with uncertainties: space-homogeneous case}

\author{
	L. Pareschi 
	\thanks{
		Mathematics and Computer Science Department,
		University of Ferrara -
		Via Machiavelli 35,
		44121 Ferrara, Italy 
		{(\sl lorenzo.pareschi@unife.it)}
	}
	\and
	M. Zanella
	\thanks{
		Mathematics Department, 
		University of Pavia - 
		Via Ferrata 5, 
		27100 Pavia, Italy 
		{(\sl mattia.zanella@unipv.it)}
	}
}

\begin{document}
\maketitle

\begin{abstract}
In this paper we propose a novel numerical approach for the Boltzmann equation with uncertainties. The method combines the efficiency of classical direct simulation Monte Carlo (DSMC) schemes in the phase space together with the accuracy of stochastic Galerkin (sG) methods in the random space. This  hybrid formulation makes it possible to construct methods that preserve the main physical properties of the solution along with spectral accuracy in the random space. The schemes are developed and analyzed in the case of space homogeneous problems as these contain the main numerical difficulties. Several test cases are reported, both in the Maxwell and in the variable hard sphere (VHS) framework, and confirm the properties and performance of the new methods. 
\medskip


\noindent{\bf Mathematics Subject Classification:}
\end{abstract}

\section{Introduction}

Kinetic equations are commonly used to describe the aggregate trends of phenomena involving large number of interacting particles. Their effectivity has been proven in rather different research fields, ranging from classical rarefied gas dynamics and granular media to socio-economy and  traffic flow engineering. Without having the ambition to review the enormous literature on this field of research, we mention \cite{Ce88,CFTV,DPR,Klar,PT1,Villani1} and the references therein for an introduction to the subject.

Despite the established literature in the deterministic description of physical phenomena, in real world applications it is of paramount importance to quantify and control possible deviations from expected trends. In particular, experimentally unavoidable uncertainties are often present due to incomplete information on the microscopic dynamics, initial states or the boundary conditions of the problem. The impact of these uncertainties must be taken into account in the model parameters and initial distributions. The general idea is based on considering those quantities as random variables that influence the evolution of the kinetic distribution and, therefore, increase the dimensionality of the problem. 

In details, let $f = f(z,v,t)$ be a nonnegative function describing the evolution of a distribution of particles travelling with velocity $v\in \mathbb R^{d_v}$ at time $t\ge 0$ and $z \in \Omega \subseteq \mathbb R^{d_z}$ is a random vector characterizing the uncertain parameters. Under the assumption of space homogeneity, the evolution of $f$ is obtained by the following Boltzmann equation 
\begin{equation}\begin{split}\label{eq:boltzmann}
\dfrac{\partial}{\partial t} f = \dfrac{1}{\varepsilon} Q(f,f), 
\end{split}\end{equation}
with initial condition $f(z,v,0) = f_0(z,v)$ that may be considered uncertain. In \eqref{eq:boltzmann} the parameter $\epsilon>0$ is the Knudsen number and $Q(\cdot,\cdot)$ defines a collision operator.  In the present work we concentrate on challenging case of the nonlinear Boltzmann operator describing binary collisions among particles
\begin{equation}
\label{eq:Q}
\begin{split}
&Q(f,f)(z,v,t) \\
&\quad = \int_{\mathbb R^{d_v}} \int_{\mathbb{S}^{d_v-1}} B(z,|v-v_*|,\omega) [f(z,v^\prime,t)f(z,v_*^\prime,t) - f(z,v,t)f(z,v_*,t)]d\omega dv_*,
\end{split}
\end{equation}
where $\omega$ is a unit vector of the sphere $\mathbb S^{d_v-1}$. As usual, in \eqref{eq:boltzmann} we adopted the notation $(v^\prime,v_*^\prime)$ to represent the collisional velocities associated to the velocities $(v,v_*)$ and collision parameter $\omega$, i.e. 
\begin{equation}\label{eq:binary}
v^\prime = \dfrac{v+v_*}{2} + \dfrac{|v-v_*|}{2}\omega, \qquad v_*^\prime = \dfrac{v+v_*}{2} - \dfrac{|v-v_*|}{2}\omega.
\end{equation}
The kernel $B(z,|v-v_*|,\omega)$ is a nonnegative function selecting the frequency of interactions and whose form is 
\[
B(z,|v-v_*|,\omega) = b_{\gamma}(\omega,z) |v-v_*|^{\gamma(z)},
\]
and $\gamma(z)$ characterizes the interaction forces between molecules. The case $\gamma = 0$ is often referred as Maxwellian case whereas $\gamma = 1$ as the hard sphere case, see \cite{Ce88} for an overview. Note that, since in a space homogeneous setting the Knudsen number acts simply as a scaling factor for the time variable, to simplify notations in the sequel we will fix ${\varepsilon}=1$.

In recent years we have seen a growing interest in building numerical methods for kinetic equations with uncertainty and studying their properties (see the collection \cite{JP}). A very well established direction of research concerns the construction of stochastic Galerkin-type methods based on the use of deterministic methods in the phase space. Such methods have demonstrated numerical and theoretical evidence of spectral accuracy in different contexts \cite{HJ,HJS,JZ,LJ,SHJ}. Their computational cost, however, is generally high due to the \emph{curse of dimensionality} present in kinetic equations and even more reinforced by the terms that model uncertainty. Moreover, the main physical properties of the solution, among which its positivity, are lost by the numerical method. Another class of methods recently developed is based on the use of control variate techniques in a multi-fidelity context \cite{DP1, DP2, LZ}. These methods, are significantly more efficient than Galerkin's stochastic approaches, especially for problems with high dimensionality of the random space. Moreover, being based on Monte Carlo collocation techniques in the parameter space their non intrusive nature permits to preserve the physical properties of the underlying deterministic numerical methods in the phase space.  We mention that, recently, multilevel Monte Carlo techniques have been also developed for some kinetic equations \cite{HPW}. 

In this work we study a different approach from the previous ones, inspired by recent particle methods based on stochastic polynomial chaos expansions for mean-field equations \cite{CPZ, CZ19}. The aim of the methods is to combine the efficiency of DSMC techniques for the nonlinear Boltzmann equation in phase space \cite{BI89, BoNa, Na83, Ba86} (see also \cite{IlNe87, Wagner1992} for convergence results) with the accuracy of stochastic Galerkin methods in parameter space \cite{HJ, Xiu}. This  novel hybrid formulation makes it possible to construct efficient methods that preserve the main physical properties of the solution along with spectral accuracy in the random space. In the present paper the schemes are developed and analyzed in the case of space homogeneous problems as they contain most of the main numerical difficulties. We stress that the approach presented here fall within the class of intrusive methods and, although we develop our methods for the classical case of rarefied gas dynamics, the derivation is quite general and admits several natural extensions to other Boltzmann type kinetic equations with uncertainties in various fields, such as Boltzmann's semiconductor equation \cite{JL2} and the Landau-Fokker-Planck equation of plasma physics \cite{HJS}. We refer also to the recent work \cite{Po} where intrusive gPC Monte Carlo methods have been derived for the linear Boltzmann equation.

The rest of the manuscript is organized as follows. First, in Section 2 we recall the classical DSMC method by Nanbu \cite{Na83, Ba86} in the case of Maxwell molecules and show how to derive its stochastic Galerking projection when uncertainties are present. 
Next Section 3 is devoted to the challenging case of variable hard sphere, where the presence of the dummy collision technique based on acceptance-rejection, introduces additional difficulties. We show how to overcome these problems with a suitable reformulation of the Monte Carlo simulation algorithm. Next, in Section 4 we test the methodology for several problems including the Kac equation, and the Boltzmann equation for Maxwell molecules and variable hard sphere. In this latter case, we show, in particular, that a suitable regularization of the acceptance-rejection technique is necessary in order to keep spectral accuracy. Some conclusions are reported in the last Section. In separate Appendices we give details on the exact solutions used as comparison and show consistency of the novel hybrid representation based on statistical samples in the phase space combined with generalized polynomial chaos expansions in the random space.

\section{DSMC-sG for Maxwellian molecules with uncertainties}\label{eq:DSMC_max}

In this section we briefly recall the basics of DSMC methods for the Boltzmann equation in the simplified case of Maxwellian molecules and the fundamentals on stochastic Galerkin techniques. Next we show how to derive the  
corresponding DSMC-sG approach.

\subsection{Classical DSMC method}\label{eq:classical}

In order to present the classical DSMC algorithm we first focus on the case without uncertainty. We are interested in the evolution of the density $f=f(v,t)$, $v\in\mathbb R^{d_v}$, $t\ge 0$, solution of \eqref{eq:boltzmann} with initial condition $f(v,0)= f_0(v)$. In the Maxwellian case, i.e. $B\equiv 1$, we may rewrite  \eqref{eq:boltzmann} as follows
\begin{equation}
\dfrac{\partial}{\partial t} f = \left[ Q^+(f,f) - \mu f\right],
\label{eq:BMM}
\end{equation}
where $Q^+(f,f)$ denotes the gain part of the collision operator, $\mu> 0$ is a constant and we assumed $\int_{\mathbb R^{d_v}}f(v_*,t)dv_* = 1$ so that $f$ is a probability density. More precisely, under the above assumptions, we have explicitly $\mu=2^{d_v-1}\pi$. 

To introduce the DSMC scheme we consider a simulation algorithm based on the time discrete form of \eqref{eq:BMM} originally proposed by Nambu \cite{Na83}.
Let us consider a time interval $[0,t_{\rm max}]$, and let us
discretize it in $\nt$ intervals of size $\Delta t$. Let us denote
by $f^n(v)$ an approximation of $f(v,n\Delta t)$. The forward
Euler scheme  writes{\[
   f^{n+1} = \left(1-{\mu\Delta t}\right)f^n +
             {\mu\Delta t}\frac{Q^+(f^n,f^n)}{\mu}.
\label{eq:forward_Euler}
\]}
Clearly if $f^n$ is a probability density both
${Q^+(f^n,f^n)}/{\mu}$ and $f^{n+1}$ are probability densities provided that {$\mu\Delta t \leq 1$}. A symmetrized version of the algorithm based on this probabilistic interpretation is reported below \cite{Na83, Ba86}.

\begin{algorithm}[Nanbu-Babovski for Maxwell molecules]~
\begin{enumerate}
\item Compute the initial velocities $ \{v_i^0, i=1,\ldots,N\} $,\\ 
      by sampling them from the initial density $f_0(v)$ 
\item   \begin{tabbing}
\= {\fP for} \= $n=0$  {\fP to} $\nt-1$ \\
          \>          \> given $\{v_i^n,i=1,\ldots,N\}$\\
          \>       \>   \ind \= \cir \= set $N_c = \IR(\mu N\Delta t/2)$ \\
          \>       \>        \> \cir \> select $N_c$ pairs $(i,j)$ uniformly among all possible pairs, \\
          \>       \>        \>      \> - \= perform the collision between $i$ and $j$, and compute \\
          \>       \>        \>      \>      \> $v_i'$ and $v_j'$ according to the  collision law \\
          \>       \>        \>      \> - set $v_i^{n+1} = v_i'$, $v_j^{n+1}=v_j'$ \\
          \>       \>        \> \cir \> set $v_i^{n+1}=v_i^n$ for all the particles that have not been selected\\
         \> {\fP end for}
\end{tabbing}
\end{enumerate}
\label{al:dsmc1}
\end{algorithm}

Here by $\IR(x)$ we denote the stochastic rounding of a
positive real number $x$
\[
   \IR(x) = \left\{\begin{array}{lll}
                     {\intg{x}} + 1 & \mbox{with probability} & x-{\intg{x}} \\
                     {\intg{x}}     & \mbox{with probability} & 1-x+{\intg{x}}
                   \end{array}
            \right.
\]
where $\intg{x}$ denotes the integer part of $x$.

The kinetic distribution, as well as its moments, is then recovered from the empirical density distribution
\begin{equation}
f_N(v,t)=\frac1{N}\sum_{i=1}^N \delta(v-v_i(t)),
\label{eq:emp}
\end{equation}
where $\delta(\cdot)$ is the the Dirac delta, or some suitable regularization of \eqref{eq:emp}.

\begin{remark}~
\begin{itemize}
\item 
The algorithm just described
can be applied to a variety of Boltzmann equation with velocity independent collision 
kernel. The only difference consists in the computation of the
collisional velocities. 
\item Note that, the method become very expensive
and practically
  unusable near the fluid regime because in this case the collision time between
  the particles becomes very small, and a huge number of collisions
  is needed in order to reach a fixed final time. Methods which overcome this kind of limitation, based on exponential time discretizations and implicit-explicit Runge-Kutta methods have been proposed in \cite{DP15, Par2, Par3}. In the sequel, even if we focus our attention on the classic DSMC algorithm described above, our methodology extends naturally also to this latter class of algorithms.
  \end{itemize} 
  \end{remark}

\subsection{DSMC-sG methods}\label{sect:DSMC}
We consider the stochastic Galerkin extension of the DSMC algorithm \ref{al:dsmc1} in presence of uncertainties  $f=f(z,v,t)$. 
%
Similarly to Section \ref{eq:classical}, in the case of Maxwell molecules $B\equiv 1$, the collision operator can be rewritten as
\begin{equation}
Q(f,f)(z,v,t)=Q^+(f,f)(z,v,t)-\mu f(z,v,t),
\label{eq:BMMi}
\end{equation}
where $\mu>0$ is a constant and we assumed $\int_{\R^{d_v}} f(z,v_*,t)\,dv_*=1$, $\forall\,z \in \Omega$.

We consider a set of $N$ samples $v_i(z,t)$, $i=1,\ldots,N$ from the kinetic solution at time $t$ and approximate $v_i(z,t)$ by its {generalized polynomial chaos (gPC) expansion}
\[
v^M_i(z,t)=\sum_{m=0}^M \hat{v}_{i,m}(t)\Phi_m(z),
\]
where  $\left\{\Phi_{m}(z)\right\}_{m=0}^M$ are a set of orthogonal polynomials, of degree less or equal to $M$ {orthonormal} with respect to the PDF $p(z)$
\[
\int_{\Omega} \Phi_n(z) \Phi_m(z) p(z)\,dz = \E[\Phi_m(\cdot)\Phi_n(\cdot)]=\delta_{mn},\qquad m,n=0,\ldots,M,
\] 
and $\hat{v}_{i,m}$ is the {projection} of the solution with respect to $\Phi_m$
\[
\hat{v}_{i,m}(t)=\int_{\Omega} v_i(z,t)\Phi_m(z) p(z)\,dz=\E[v_i(\cdot,t)\Phi_m(\cdot)].
\]
We underline that the generation of the set of stochastic samples $v_i(z,t)$ is a problem in itself different from the standard particle generation in DSMC methods. Although this aspect is fundamental for the practical application of the method, to simplify the presentation, we have postponed the details of the approach used in Appendix B.1. Similarly, the convergence properties of the resulting gPC expansion of the samples are analyzed in Appendix B.2.

To define the DSMC-sG algorithm we consider the projection on the above space of the collision process in the DSMC method. In the case of the uncertain Boltzmann collision term \eqref{eq:BMMi} we have
\begin{eqnarray*}
v_i'(z,t) &=& \frac12(v_i(z,t)+v_j(z,t))+\frac12 |v_i(z,t)-v_j(z,t)|\omega,\\
v_j'(z,t) &=& \frac12(v_i(z,t)+v_j(z,t))-\frac12 |v_i(z,t)-v_j(z,t)|\omega.
\end{eqnarray*}
Let us observe that
\be
|v_i'(z,t)-v_j'(z,t)| =  |v_i(z,t)-v_j(z,t)|,
\label{eq:mrv}
\ee
so that the modulus of the relative velocity is unchanged during collisions. 

We first substitute the velocities by their gPC expansion   
\begin{eqnarray*}
{v^{M}_i}^\prime(z,t) &=& \frac12(v^M_i(z,t)+v^M_j(z,t))+\frac12 |v^M_i(z,t)-v^M_j(z,t)|\,\omega,\\
{v^{M}_i}^\prime(z,t) &=& \frac12(v^M_i(z,t)+v^M_j(z,t))-\frac12 |v^M_i(z,t)-v^M_j(z,t)|\,\omega
\end{eqnarray*}
and then project by integrating against $\Phi_m(z)\,p(z)$ on $\Omega$ to get for $m=0,\ldots,M$
\begin{eqnarray}
\label{eq:c1}
\hat{v}_{i,m}^{\prime}(t) &=& \frac12(\hat{v}_{i,m}(t)+\hat{v}_{j,m}(t))+\frac12 \hat V^m_{ij}\,\omega,\\
\label{eq:c2}
\hat{v}_{j,m}^{\prime}(t) &=& \frac12(\hat{v}_{i,m}(t)+\hat{v}_{j,m}(t))-\frac12 \hat V^m_{ij}\,\omega
\end{eqnarray}
where 
\be
\hat V^m_{ij}=\int_{\Omega} |v^M_i(z,t)-v^M_j(z,t)|\Phi_m(z)\, p(z)\,dz,
\label{eq:cmt}
\ee
for each $m=0,\ldots,M$, are collision matrices consisting of $N^2$ elements.

The sG extensions of the DSMC algorithms by Nanbu for Maxwell molecules is reported below.

\begin{algorithm}[DSMC-sG for Maxwell molecules]~
\begin{enumerate}
\item Compute the initial gPC expansions $ \{v^{M,0}_i, i=1,\ldots,N\} $,\\ 
      from the initial density $f_0(v)$
\item   \begin{tabbing}
\= {\fP for} \= $n=0$  {\fP to} $\nt-1$ \\
          \>          \> given $\{\hat v_i^{m,n},i=1,\ldots,N,\,m=0,\ldots,M\}$\\
          \>       \>   \ind \= \cir \= set $N_c = \IR(\mu N\Delta t/2)$ \\
          \>       \>        \> \cir \> select $N_c$ pairs $(i,j)$ uniformly among all possible pairs, \\
          \>       \>        \> \cir \> Compute the collision matrices $\hat V^m_{ij}$,  $m=0,\ldots,M$,\\
\>       \>        \>  \> for all $N_c$ collision pairs using \eqref{eq:cmt}\\ 
          \>       \>        \>      \> - \= perform the collision between $i$ and $j$, and compute \\
          \>       \>        \>      \>      \> $\hat v_{i,m}^{\prime}$ and $\hat v_{j,m}^{\prime}$ according to \eqref{eq:c1}-\eqref{eq:c2} \\
          \>       \>        \>      \> - set $\hat v_{i,m}^{n+1} = \hat v_{i,m}^{\prime}$, $\hat v_{j,m}^{n+1}=\hat v_{j,m}^{\prime}$ \\
          \>       \>        \> \cir \> set $\hat v_{i,m}^{n+1}=\hat v_{i,m}^{n}$ for all the particles that have not been selected\\
         \> {\fP end for}
\end{tabbing}
\end{enumerate}
\end{algorithm}

\begin{remark}~
\begin{itemize}
\item The DSMC-sG algorithms just described extends naturally to other DSMC methods, including Bird's method \cite{Bi76}. We omit the details for brevity.
\item 
Note that, since $i$ and $j$ are selected at random for a total of $N_c \leq N/2$ pairs, we do not need all $N^2$ elements in the collision matrix \eqref{eq:cmt} to evaluate the collision process. For fixed values of $i$ and $j$ we approximate the vector $\hat V^m_{ij}$ by Gauss quadrature 
\be
\hat V^m_{ij}(t) \approx \sum_{h=0}^H w_h |v^M_i(z_h,t)-v^M_j(z_h,t)|\Phi_m(z_h).
\ee 
The resulting scheme requires $O(MH)$ operations to compute $v^M_i(z_h,t)$ and $v^M_j(z_h,t)$ for all $h$'s and $O(MH)$ operations to evaluate $\hat V^m_{ij}(t)$ for all $m$'s. Taking $H=M$ the total cost of a Monte Carlo collision at each time step is therefore $O(M^2)$, and the total cost of a DSMC-sG time step is $O(M^2 N_c)$. 
\end{itemize}
\label{rk:1}
\end{remark}

\section{DSMC-sG for hard spheres with uncertainties}
The DSMC methods have to be modified when the scattering cross
section is not constant. First we describe the so called \emph{dummy collision} technique \cite{Bi76, Na83, Par2} in the 
deterministic setting and then we discuss how to extend the DSMC-sG methodology to this general case. 

\subsection{Classical DSMC method} To this aim let us first consider the deterministic case 
for variable hard spheres (VHS) molecules where the collision kernel has the form
\be
B(z,v,v_*,\omega)=B(\vert v
- v_{\ast} \vert),
\ee
and satisfies a cut-off hypothesis, which is
essential from a numerical point of view.

We will denote by $Q_\Sigma(f,f)$ the collision operator obtained
by replacing the kernel $ B$ with the kernel $ B_\Sigma${\[
 B_\Sigma(\vert v - v_ {\ast}\vert) =\min\left\{ B(\vert v
-
  v_{\ast} \vert),\Sigma\right\},\quad \Sigma > 0.
\]}
Thus, for a fixed $\Sigma$, let us consider the homogeneous
problem {\[ {{\partial f}\over {\partial t}} =
 Q_\Sigma(f, f). \label{eq:HOM3}
\]}
The operator $Q_\Sigma(f, f)$ can be written in the form
$P(f,f)-\mu f$ taking {\[ P(f,f)= Q^{+}_\Sigma(f,f) +
f(v)\int_{\RR^{d_v}} \int_{S^{d_v-1}} [\Sigma- B_\Sigma(\vert v - v_{\ast}
\vert)] f(v_{\ast})\, d\omega\,dv_{\ast}, \label{eq:PB}
\]}
with $\mu=2^{d_v-1}\pi\Sigma$ and
{\[ Q^{+}_\Sigma(f,f)=
 \int_{\RR^{d_v}} \int_{S^{d_v-1}}
 B_\Sigma(\vert v - v_{\ast} \vert)
f(v^{\prime})f(v_{\ast}^{\prime})\, d\omega\,dv_{\ast}.
\]}
The generalization of the DSMC scheme is obtained by
using the acceptance-rejection technique to sample the post
collisional velocity according to $P(f,f)/\mu$.

\begin{algorithm}[Nanbu-Babovski for VHS molecules]~
\begin{enumerate}
\item Compute the initial velocities $ \{v_i^0, i=1,\ldots,N\} $,\\ 
      by sampling them from the initial density $f_0(v)$ 
\item   \begin{tabbing}
\= {\fP for} \= $n=0$  {\fP to} $\nt-1$ \\
          \>          \> given $\{v_i^n,i=1,\ldots,N\}$\\
          \>       \>   \ind \= \cir \= compute an upper bound $\Sigma$ of the cross section \\
          \>       \>   \ind \= \cir \= set $\mu=2^{d-1}\pi\Sigma$ and $N_c = \IR(\mu N\Delta t/2)$ \\
                    \>       \>        \> \cir \> select $N_c$ dummy collision pairs $(i,j)$ uniformly  \\
          \>       \>        \>      \> among all possible pairs, and for those \\
          \>       \>        \>      \> - \= compute the relative cross section $ B_{ij} =  B(|v_i-v_j|)$ \\
          \>       \>        \>      \> - \> {\fP if} $\Sigma\,  \rand <  B_{ij}$, $\xi$ uniform in $(0,1)$ \\
          \>       \>        \>      \>  \ind \= perform the collision between $i$ and $j$, and compute \\
          \>       \>        \>      \>      \> $v_i'$ and $v_j'$ according to the  collision law \\
          \>       \>        \>      \> \>  set $v_i^{n+1} = v_i'$, $v_j^{n+1}=v_j'$ \\
          \>       \>        \> \cir \> set $v_i^{n+1}=v_i^n$ for all the particles that have not been selected\\
         \> {\fP end for}
\end{tabbing}
\end{enumerate}
\end{algorithm}

The upper bound $\Sigma$ should be chosen as small as possible, to
avoid inefficient rejection, and it should be computed fast. It is
be too expensive to compute $\Sigma$ as {\[
   \Sigma =  B_{\mbox{\small max}} \equiv \max_{ij}  B(|v_i-v_j|),
\]}
since this computation would require an $O(N^2)$ operations.
An upper bound of $  B_{\mbox{\small max}}$ is obtained by taking
$\Sigma  =  B(2\Delta v)$, where {\[ \Delta v = \max_i
|v_i-\bar{v}|, \quad \bar{v}:=\frac1N\sum_i v_i.
\]}

\subsection{DSMC-sG for VHS molecules}\label{sect:VHS}

The extension of the DSMC-sG algorithm to the VHS case is not straightforward due to the acceptance-rejection process which in general depends on the random parameter $z$. In the sequel we will consider
\be
B(z,v,v_*,\omega)=B(z,\vert v
- v_{\ast} \vert),
\ee
and denote by $Q_\Sigma(f,f)$ the collision operator obtained
by replacing the kernel $ B$ with the kernel $ B_\Sigma${\[
 B_\Sigma(z,\vert v - v_ {\ast}\vert) =\min\left\{ B(z,\vert v
-
  v_{\ast} \vert),\Sigma\right\},\quad \Sigma > 0.
\]}
Given a random number $\rand$ uniform in $(0,1)$, we rewrite the acceptance-rejection collision process in the equivalent form
\begin{equation}
\label{eq:binary_VHS}
\begin{split}
v_i'(z,t) &= v_i(z,t) - \frac12\Psi(\Sigma\,  \rand <  B_{ij}(z))\left((v_i(z,t)-v_j(z,t))- |v_i(z,t)-v_j(z,t)|\omega\right),\\
v_j'(z,t) &= v_j(z,t)+ \frac12\Psi(\Sigma\,  \rand <  B_{ij}(z))\left((v_i(z,t)-v_j(z,t))- |v_i(z,t)-v_j(z,t)|\omega\right),
\end{split}
\end{equation} 
where $\Psi(\cdot)$ is the indicator function and
\[
B_{ij}(z)=B(z,|v_i(z,t)-v_j(z,t)|).
\]
Note that \eqref{eq:mrv} still holds true. Thanks to the new formulation, we can perform the projection on the space of modes in the gPC expansion to get for $m=0,\ldots,M$
\begin{eqnarray}
\label{eq:c1v}
\hat{v}_{i,m}^{\prime}(t) &=& \hat{v}_{i,m}(t)-\frac12 \hat W^m_{ij}(\rand)+\frac12 \hat V^m_{ij}(\rand)\,\omega,\\
\label{eq:c2v}
\hat{v}_{j,m}^{\prime}(t) &=& \hat{v}_{j,m}(t)+\frac12 \hat W^m_{ij}(\rand)-\frac12 \hat V^m_{ij}(\rand)\,\omega,\end{eqnarray}
where now
\begin{eqnarray}
\label{eq:cmt1}
\hat W^m_{ij}(\rand) &=& \int_{\Omega} \Psi(\Sigma\,  \rand <  B_{ij}(z))\left(v^M_i(z,t)-v^M_j(z,t)\right) \Phi_m(z) p(z)\,dz,\\
\label{eq:cmt2}
\hat V^m_{ij}(\rand) &=& \int_{\Omega} \Psi(\Sigma\,  \rand <  B_{ij}(z))|v^M_i(z,t)-v^M_j(z,t)| \Phi_m(z) p(z)\,dz.
\end{eqnarray}
The above quantities are computed at each collision for a given $i$, $j$ and $\rand$. Using Gauss quadrature we have
\begin{eqnarray*}
\hat W^m_{ij}(\rand) &\approx & \sum_{h=0}^H w_h \Psi(\Sigma\,  \rand <  B^M_{ij}(z_h)) \left(v^M_i(z_h,t)-v^M_j(z_h,t)\right)\Phi_m(z_h)\\
\hat V^m_{ij}(\rand) &\approx & \sum_{h=0}^H w_h \Psi(\Sigma\,  \rand <  B^M_{ij}(z_h)) |v^M_i(z_h,t)-v^M_j(z_h,t)|\Phi_m(z_h),
\end{eqnarray*} 
where
\be
B^M_{ij}(z_h) = B(z,|v^M_i(z,t)-v^M_j(z,t)|).
\ee
Note that, similarly to Remark \ref{rk:1}, for given $i$, $j$ and $\rand$ the quantities $\hat W^m_{ij}(\rand)$ and $\hat V^m_{ij}(\rand)$ can be computed at a cost $O(MH)$. 

We can summarize the DSMC-sG method for VHS molecules as follows.

\begin{algorithm}[DSMC-sG for VHS molecules]~
\begin{enumerate}
\item Compute the initial gPC expansions $ \{v^{M,0}_i, i=1,\ldots,N\} $,\\ 
      from the initial density $f_0(v)$      
\item   \begin{tabbing}
\= {\fP for} \= $n=0$  {\fP to} $\nt-1$ \\
          \>       \> given $\{\hat v_{i,m}^{n},i=1,\ldots,N,\,m=0,\ldots,M\}$\\
          \>       \>   \ind \= \cir \= compute an upper bound $\Sigma$ of the cross section \\
          \>       \>   \ind \= \cir \= set $\mu=2^{d-1}\pi\Sigma$ and $N_c = \IR(\mu N\Delta t/2)$ \\
          \>       \>        \> \cir \> select $N_c$ dummy collision pairs $(i,j)$ uniformly  \\
          \>       \>        \>      \> among all possible pairs and for those \\
          \>       \>        \>      \> - \= Select $\rand$ uniformly in $(0,1)$\\
          \>       \>        \>      \> - \= Compute the collision matrices $\hat W^m_{ij}(\rand)$, $\hat V^m_{ij}(\rand)$, $i,j=1,\ldots,N$,\\
          \>       \>        \>      \>   \> $m=0,\ldots,M$, using \eqref{eq:cmt1}-\eqref{eq:cmt2}. \\
          \>       \>        \>      \> - \= perform the dummy collision between $i$ and $j$, computing \\
          \>       \>        \>      \>      $\hat v_{i,m}^{\prime}$ and $\hat v_{j,m}^{\prime}$ according to \eqref{eq:c1v}-\eqref{eq:c2v} \\
          \>       \>        \>      \> - set $\hat v_{i,m}^{n+1} = \hat v_i^{\prime}$, $\hat v_{j,m}^{n+1}=\hat v_{j,m}^{\prime}$ \\
          \>       \>        \> \cir \> set $\hat v_{i,m}^{n+1}=\hat v_{i,m}^{n}$ for all the particles that have not been selected\\
         \> {\fP end for}
\end{tabbing}
\end{enumerate}
\end{algorithm}

In the reformulation of the acceptance-rejection algorithm we introduced the indicator function $\Psi(\cdot)$. Nevertheless, in the stochastic Galerkin framework we obtain spectral approximation only for smooth functions. Therefore, the microscopic binary collision term \eqref{eq:binary_VHS} may deteriorate the overall convergence of the DSMC-sG scheme.  In order to overcome this problem one can consider the following regularization 
\begin{equation}
\label{eq:binary_sig}
\begin{split}
v_i'(z,t) &= v_i(z,t) - \frac12 K( \beta(\Sigma\,  \rand-B_{ij}(z)))\left((v_i(z,t)-v_j(z,t))- |v_i(z,t)-v_j(z,t)|\omega\right),\\
v_j'(z,t) &= v_j(z,t)+ \frac12 K(\beta(\Sigma\,  \rand-B_{ij}(z)))\left((v_i(z,t)-v_j(z,t))- |v_i(z,t)-v_j(z,t)|\omega\right),
\end{split}
\end{equation} 
where $K(\beta(\cdot))$ is a sigmoid function dependent on the parameter $\beta\gg0$. These regularized acceptance-rejection collision process \eqref{eq:binary_sig}, however, introduces a dissipation of the energy. To overcome this issue and keep exact energy and momentum conservation one can couple \eqref{eq:binary_sig} with a thermalization process of the form
\begin{equation}\label{eq:binary_terma}
v_i^{\prime\prime}(z,t) = (v_i^{\prime}(z,t) - u(z)) \sqrt{\dfrac{T(z)}{T^\prime(z)}}+u(z), \qquad
v_j^{\prime\prime}(z,t) = (v_j^{\prime}(z,t) -u(z)) \sqrt{\dfrac{T(z)}{T^\prime(z)}}+u(z),
\end{equation}
being $u(z)$ the mean velocity and $T(z), T^\prime(z)$ the pre-collision and post-collision temperatures.


\begin{remark}
The generalization of the above DSMC-sG algorithm for VHS to other Monte Carlo methods is not straightforward as in the case of Maxwell molecules. In particular, in the case of Bird's algorithm we must deal with the additional difficulty of a local time step which depends from $z$. Here we will not discuss further this aspect that will be the subject of future investigations.  
\end{remark}

\section{Numerical tests}\label{sect:num}
In this section we present several numerical tests for the novel DSMC-sG algorithms applied to classical space homogeneous collisional kinetic equations of the type \eqref{eq:boltzmann}-\eqref{eq:Q}. In more details, we first compare numerical and exact solutions of the Kac model and of the 2D model for Maxwellian molecules with uncertainties. Next, we consider the case of hard spheres and compare the performance of the scheme in several benchmark tests. 

\subsection{Test 1: Kac model}\label{sect:kac}

\begin{figure}[htb]
\centering
\subfigure[$t = 0$]{
\includegraphics[scale = 0.25]{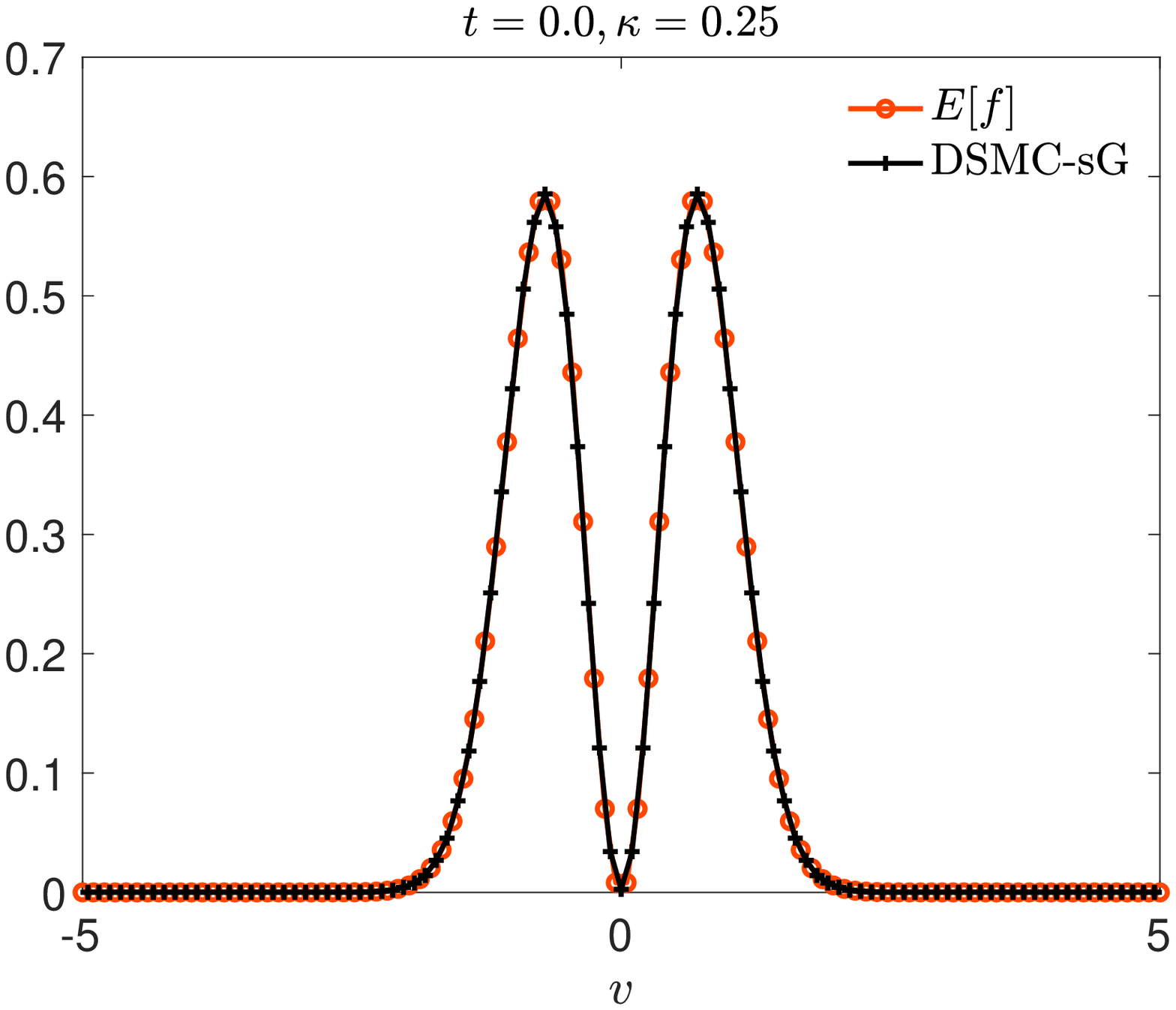}}
\subfigure[$t=1$]{
\includegraphics[scale = 0.25]{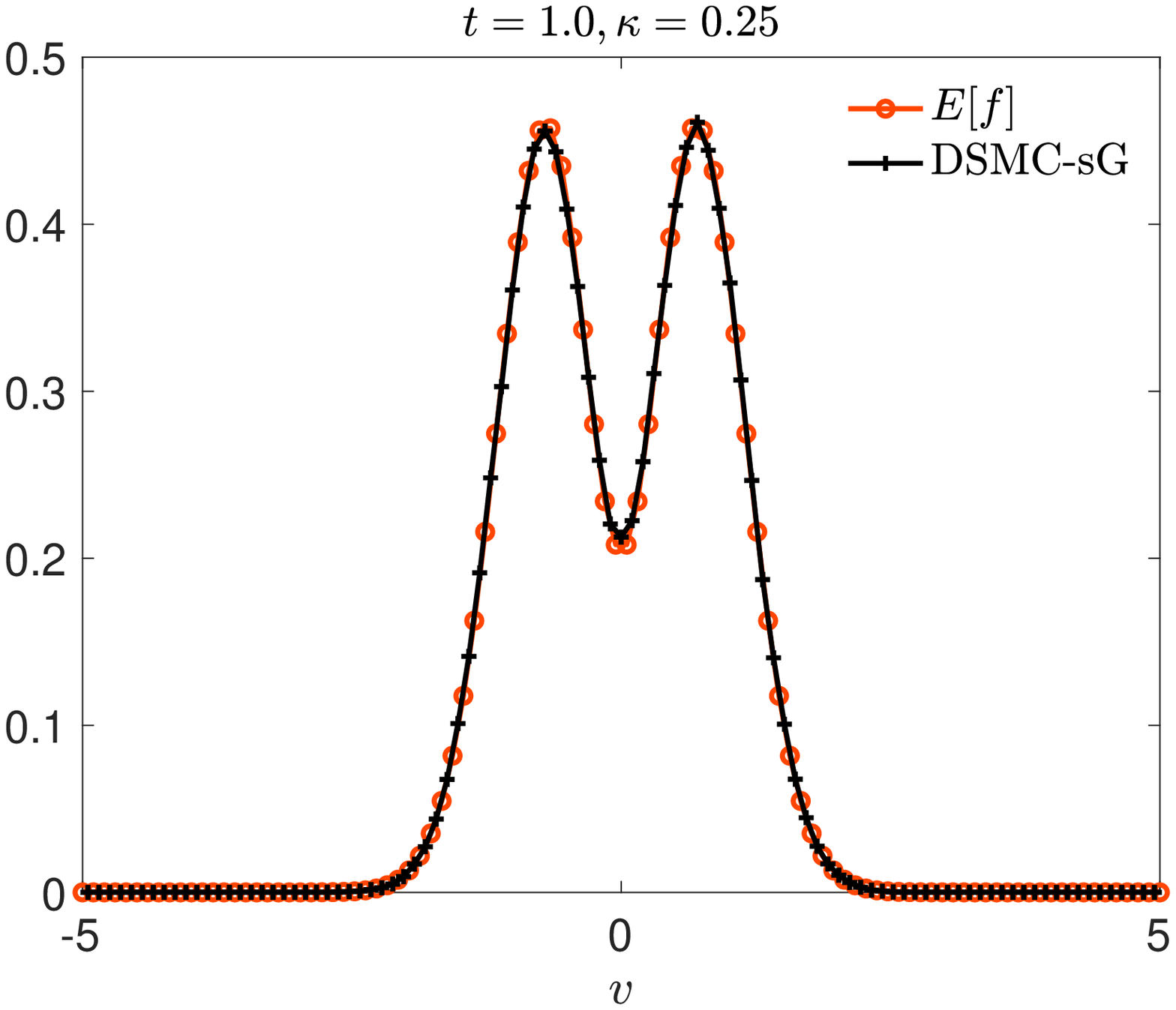}}
\subfigure[$t=5$]{
\includegraphics[scale = 0.25]{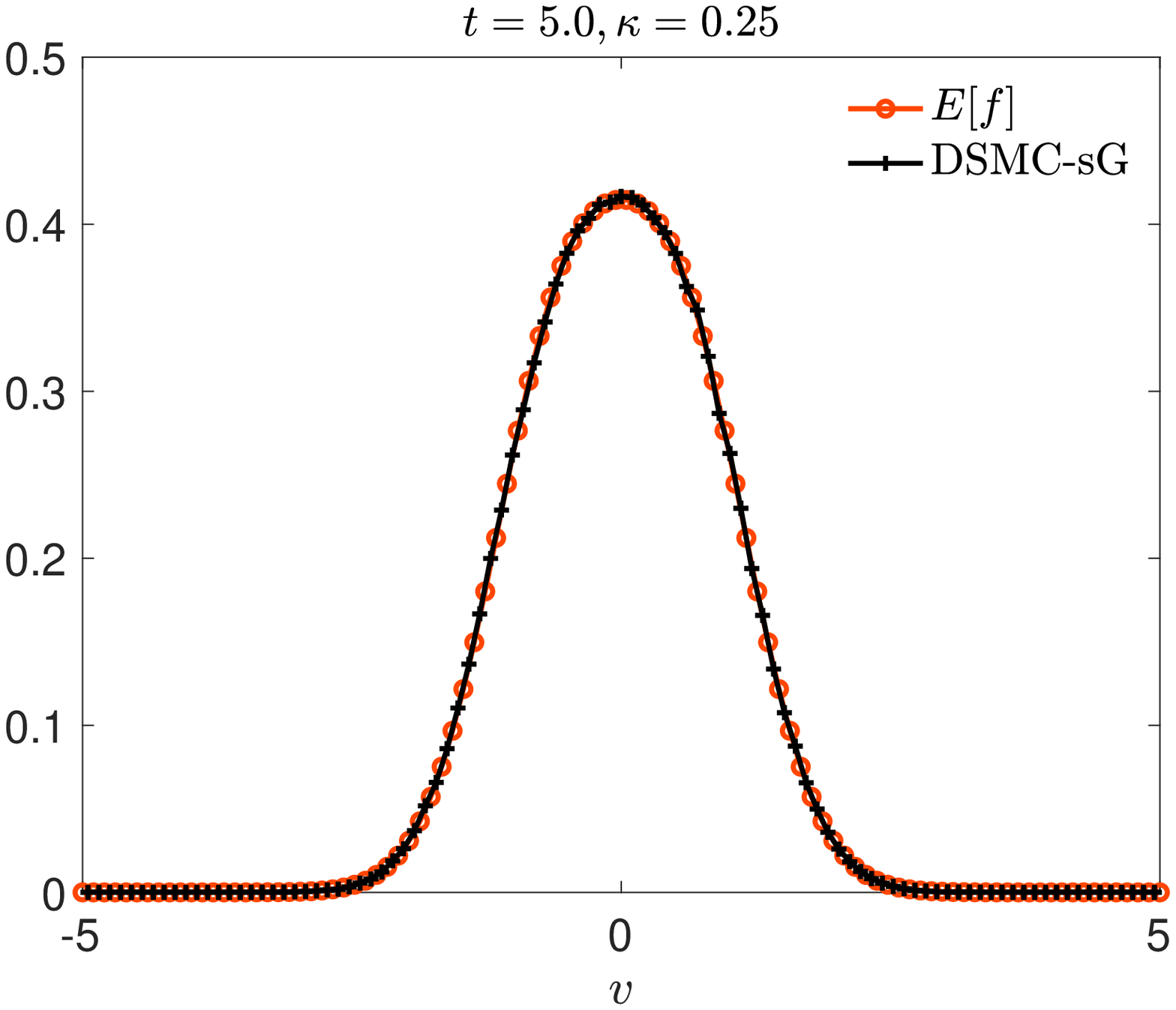}} 
\subfigure[$t=0$]{
\includegraphics[scale = 0.25]{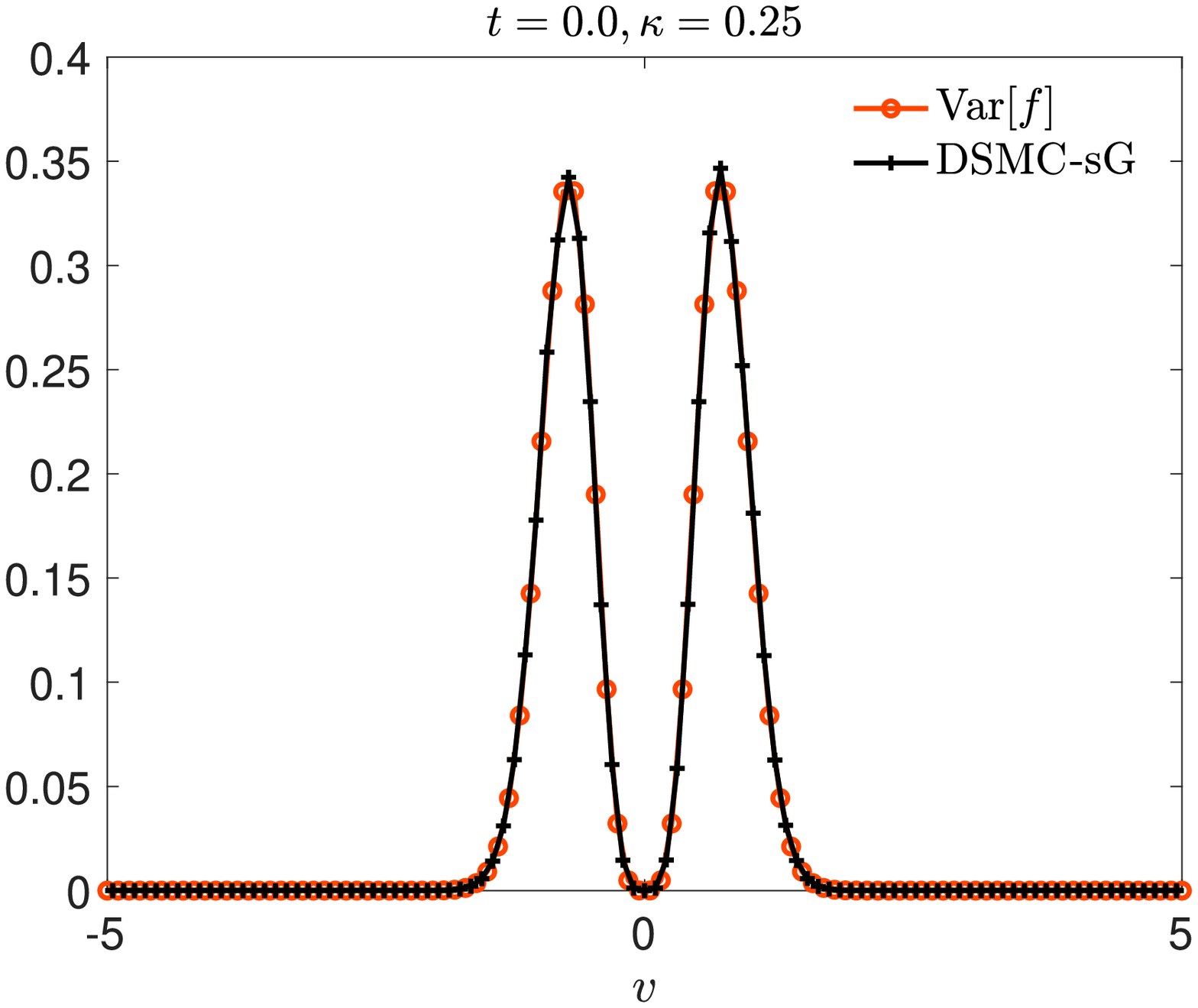}}
\subfigure[$t=1$]{
\includegraphics[scale = 0.25]{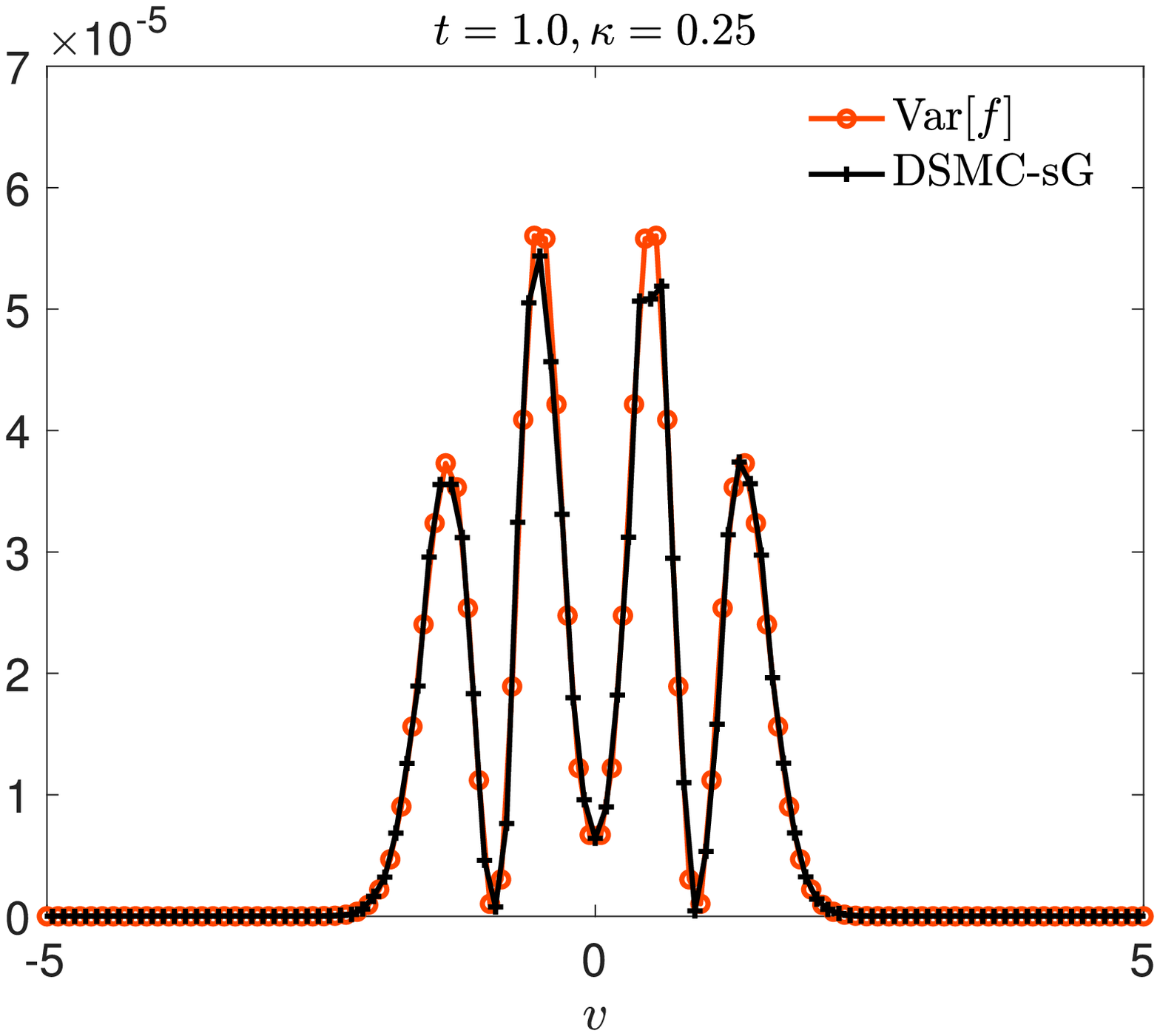}}
\subfigure[$t=5$]{
\includegraphics[scale = 0.25]{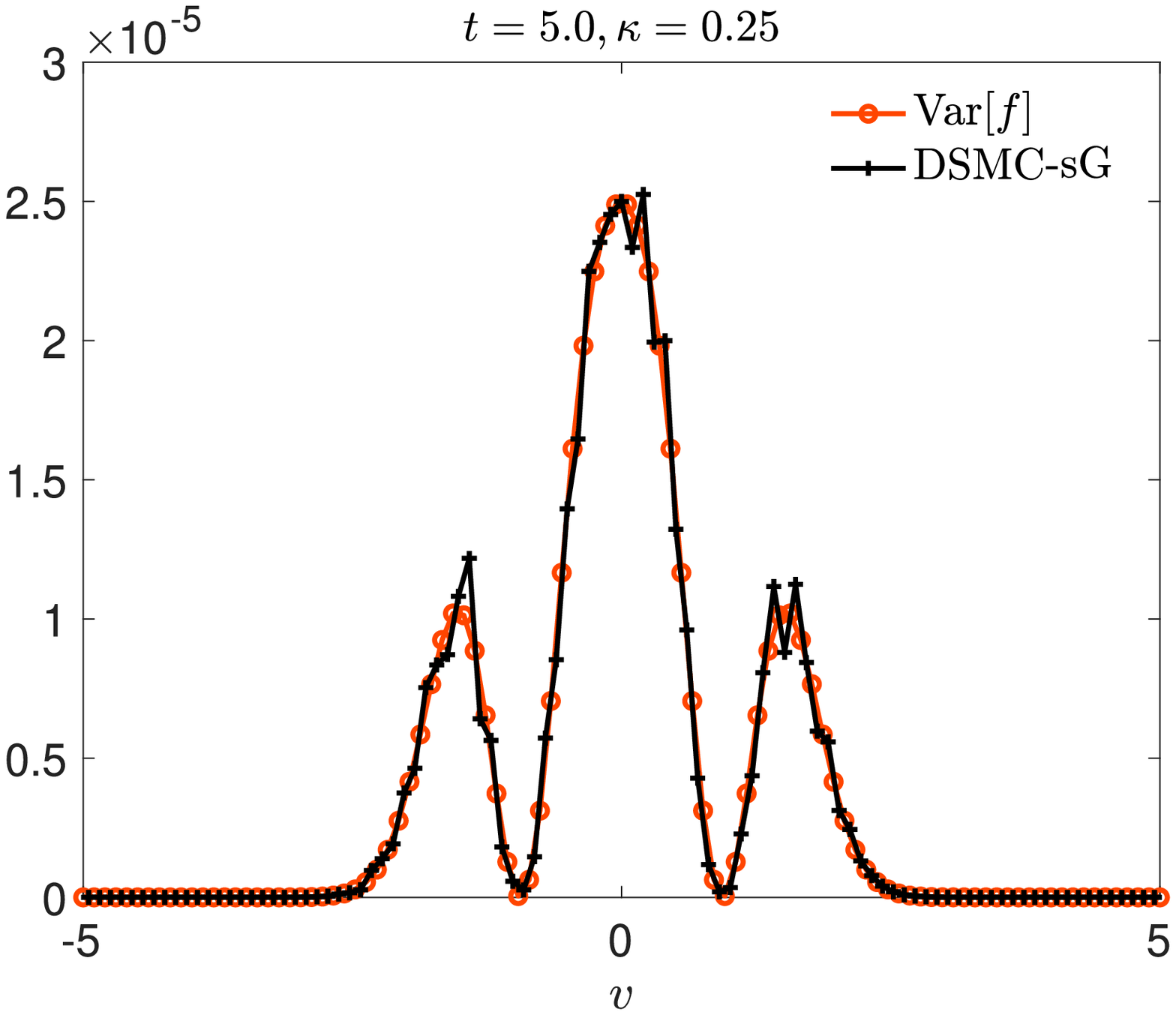}}
\caption{\textbf{Test 1}. Evolution of the exact and reconstructed expected solution (top row) and of its variance (bottom row) of the Kac model \eqref{eq:kac} with random initial temperature \eqref{eq:t0} and $\alpha(z)$ like in \eqref{eq:alpha} with $\kappa = 0.25$. We considered $N = 10^6$ particles and $M=5$ Galerkin modes whereas the time discretization of $[0, 5]$ with  $\Delta t = 10^{-1}$.  The density reconstruction in the velocity space has been performed in the interval $[-5,5]$ with $N_v = 100$ gridpoints. }
\label{fig:kac1}
\end{figure}

Let us consider first the Kac model
\begin{equation}\label{eq:kac}
\frac{\partial}{\partial t} f(z,v,t)= \dfrac{1}{2\pi} \int_0^{2\pi}\int_{\mathbb R} \left[ f(z,v^\prime,t)f(z,v_*^\prime,t)-f(z,v,t)f(z,v_*,t) \right]dv_* d\omega
\end{equation}
with binary interactions given by 
\begin{equation}
\label{eq:binary_kac}
\begin{split}
v^\prime(z) &= v(z)\cos\omega -v_*(z)\sin\omega \\
v_*^\prime(z) &= v(z)\sin\omega+ v_*(z)\cos\omega
\end{split}
\end{equation}
and $v,v_*\in \mathbb R$, $\omega \in [0,2\pi]$. In the following, we will consider an uncertain initial temperature of the system. Hence, we consider the uncertain initial distribution
\begin{equation}\label{eq:f0_kac}
f_0(z,v) = \alpha(z)\sqrt{\alpha(z)}v^2 e^{ - \alpha(z) v^2},
\end{equation}
where in particular we chose 
\begin{equation}
\label{eq:alpha}
\alpha(z) = 2 + \kappa z, \qquad z\sim \mathcal U([-1,1]). 
\end{equation}
It can be verified that mass and energy are conserved whereas momentum is not in general since it decays to zero. Anyway, since $f_0$ in \eqref{eq:f0_kac} is already centered the evolution of $f(z,v,t)$ solution of the Kac model conserves also the mean velocity defined at time $t = 0$ as $m = \int_{\mathbb R}vf_0(z,v)dv$. Therefore, mass and momentum are not dependent on the introduced uncertainty whereas the temperature reads
\begin{equation}
\label{eq:t0}
T(z) = \dfrac{1}{2} \int_{\mathbb R} v^2 f_0(z,v)dv = \dfrac{3\sqrt{\pi}}{8\alpha(z)}.
\end{equation}
Under the above assumptions we can analytically obtain the exact solution of such model as
\[
f(z,v,t) = (A(z,t) + B(z,t)v^2) e^{-s(z,t)v^2},
\]
where $A(z,t)$, $B(z,t)$, and $s(z,t)$ are given in \eqref{app:system} and  \eqref{app:s}, see Appendix \ref{app:kac} for further details. In order to test the DSMC-sG scheme we compare the reconstructed expected density with the quantity $\mathbb E[f(z,v,t)]$ which can be computed at each time step. 

\begin{figure}
\includegraphics[scale = 0.4]{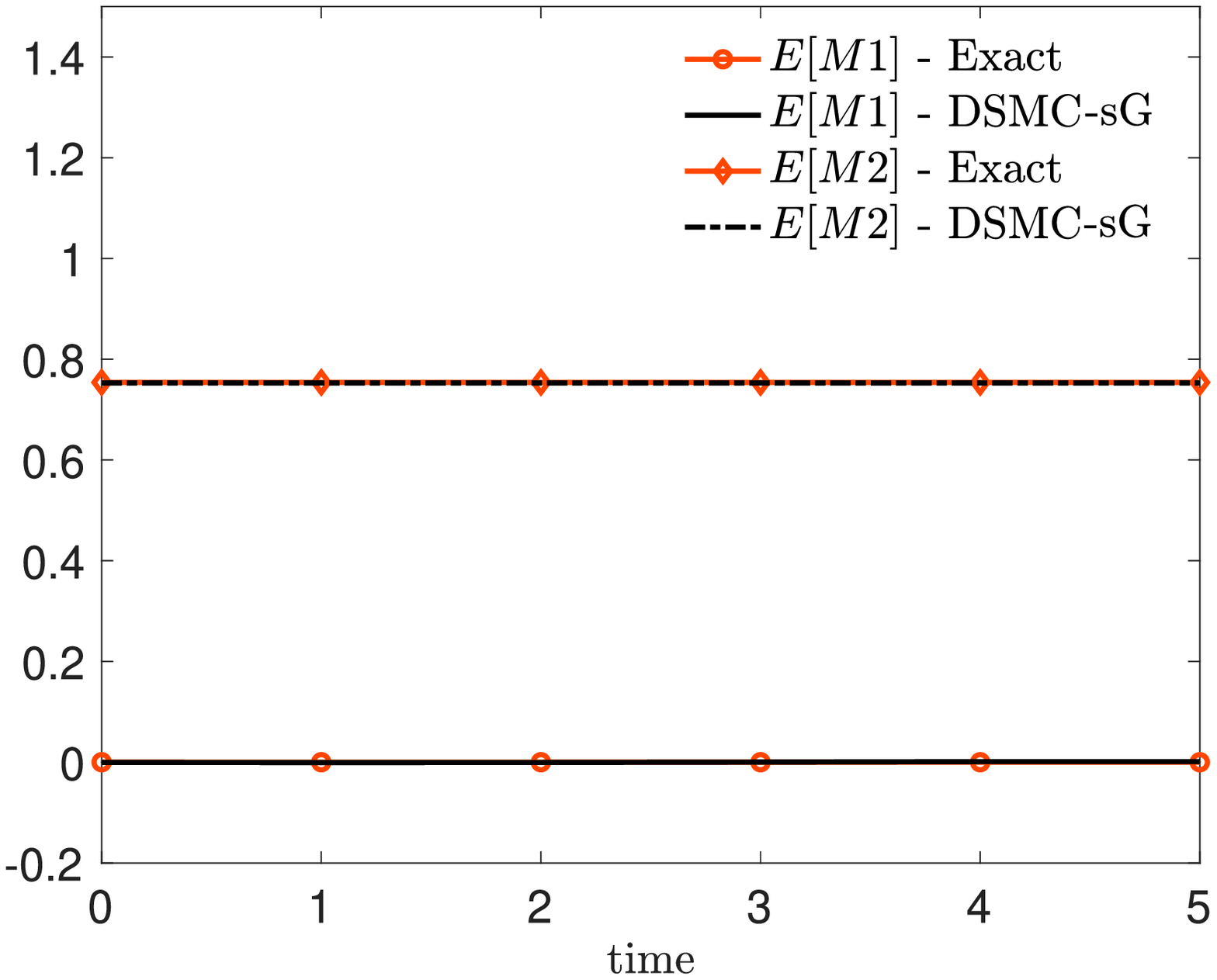}
\includegraphics[scale = 0.4]{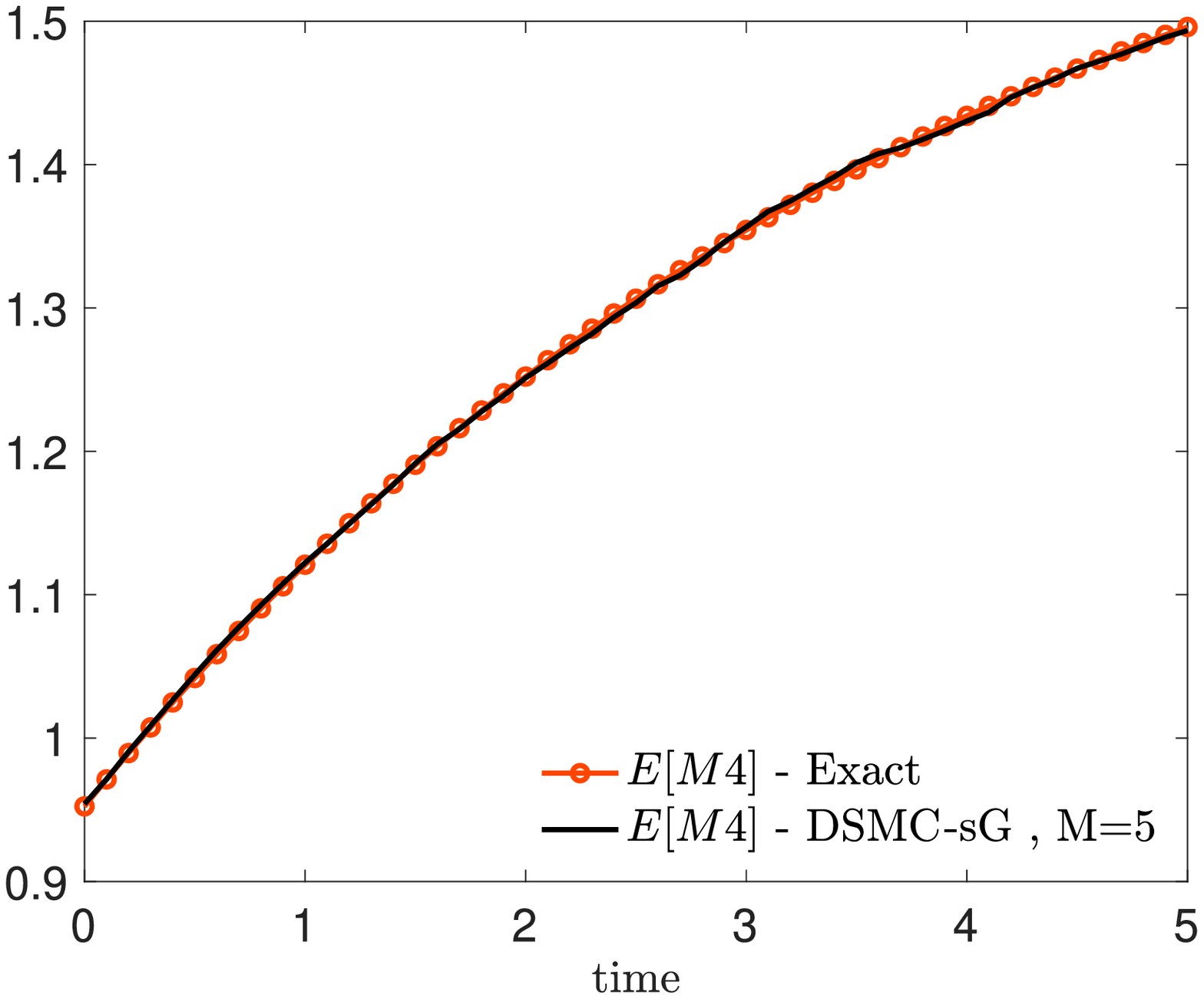}
\caption{\textbf{Test 1}. Left: evaluation of momentum and temperature from expected exact solution and the particle one. Right: evolution of the expected fourth order moment $\mathbb E[M4]$, see \eqref{eq:M4_def}. The moments are extrapolated from the evolution described in Figure \ref{fig:kac1}. }
\label{fig:kac2}
\end{figure}

In Figure \ref{fig:kac1} we report the evolution of the expected exact solution $\mathbb E[f]$ and of its reconstruction through the introduced DSMC-sG methods. In details, at the computational level we considered $N=10^6$ particles for the DSMC algorithm and $M = 5$ for the stochastic Galerkin projection. The projection of the samples from the initial distribution has been performed as described in Appendix \ref{app:particle}. Furthermore, we considered $\alpha(z)$ as in \eqref{eq:alpha} with $\kappa = 0.25$ and  $M = 5$ Galerkin modes for each particle of the method. The evolution is reported with the kinetic density reconstructed by considering the histogram on a grid in $[-5,5]$ with $N_v = 100$ gridpoints and $\Delta t = 10^{-1}$. We can clearly observe an excellent agreement in time of the DSMC-sG solution with the expected density solution of the Kac model \eqref{eq:kac}-\eqref{eq:binary_kac}.

Next, we define the $k$-order moment of the distribution $f(z,v,t)$ as
\begin{equation}
\label{eq:M4_def}
Mk(z,t) = \int_{\mathbb R} v^k f(z,v,t)dv, \qquad k \in\mathbb N, 
\end{equation}
whose approximation can be obtained at the particle gPC level as
\[
Mk^M(t) \approx \dfrac{1}{N} \sum_{i = 1}^N (v_i^M)^k(z,t),
\]
and thus, its expectation reads
\[
\mathbb E[Mk](t) \approx \dfrac{1}{N} \sum_{i = 1}^N \int_{\Omega}(v_i^M)^k(z,t)p(z)dz.
\]
In Figure \ref{fig:kac2} we represent the evolution of the expected first, second and fourth order moments, respectively $\mathbb E[M1]$, $\mathbb E[M2]$ and $\mathbb E[M4]$. We plot their evolution obtained from the exact solution of the model and its approximation through DSMC-sG scheme with $N = 10^6$ particles and $M = 5$ degree polynomials for the stochastic Galerkin approximation. It can be easily observed how the DSMC-sG method preserves exactly the conserved quantities and captures very well the evolution of the fourth order moment.

Finally, in order to show the spectral convergence property of the scheme we consider a reference DSMC-sG evolution of $M4^M(z,t)$ obtained with $N = 10^6$, $\Delta t = 0.1$ and stochastic Galerkin scheme up to order $M = 25$. Hence, if we store the collisional tree generating the reference solution, we may check the $L^2$ convergence of the DSMC-sG algorithm. In Figure \ref{fig:kac3} we present the decay of the $L^2 $ error for increasing $M$ obtained from the initial distribution \eqref{eq:f0_kac} with $\alpha(z)$ like in \eqref{eq:alpha}, where $\kappa = 0.25$ and $\kappa= 0.75$. In details, the left figure we show the error produced by the DSMC-sG algorithm at fixed time $T = 5$ whereas, in the right figure, we show in the semilogarithmic scale the obtained error in the whole time interval $[0,5]$. The spectral accuracy of the DSMC-sG approach in the parameter space appears clearly from these numerical results.

\begin{figure}
\includegraphics[scale = 0.4]{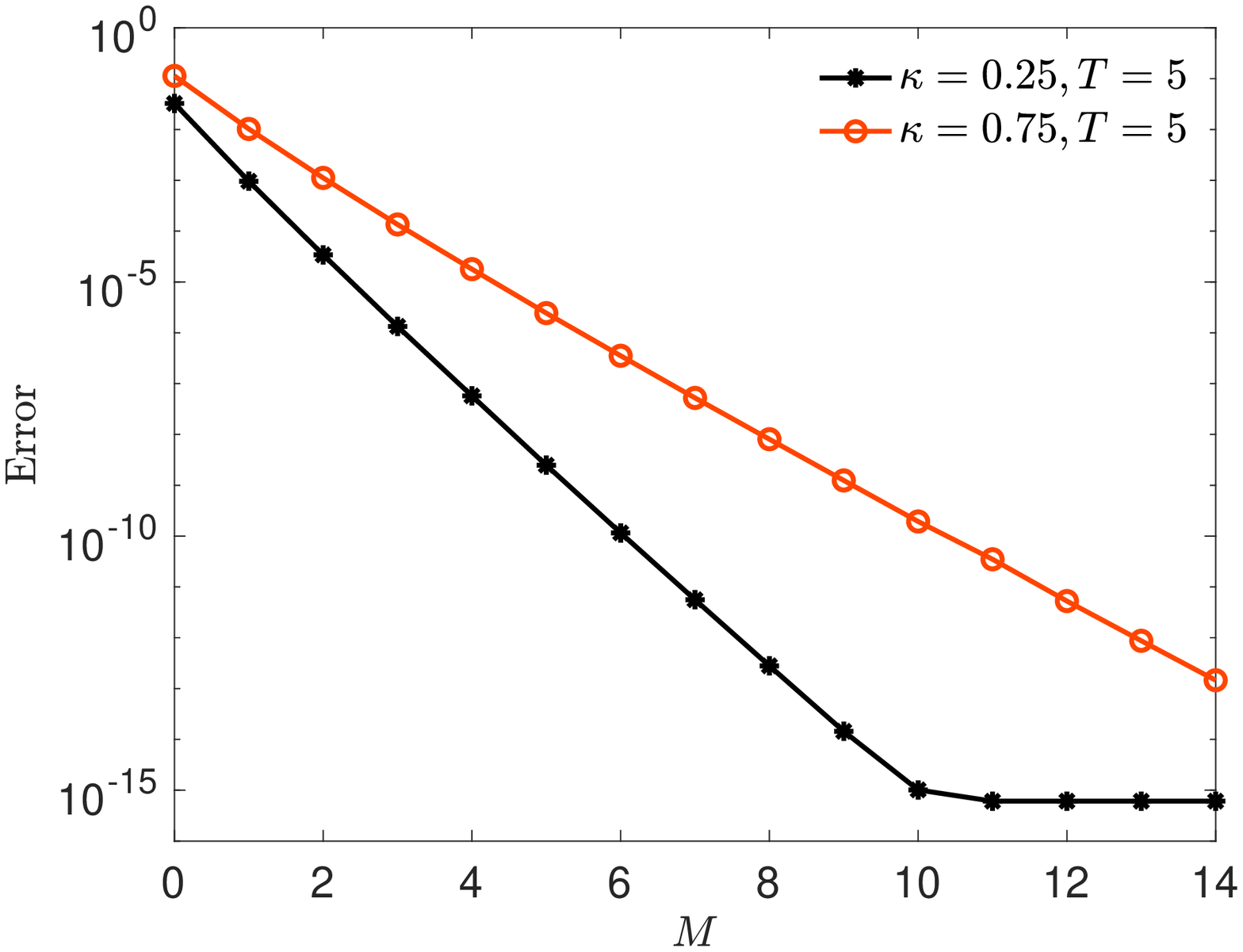}
\includegraphics[scale = 0.4]{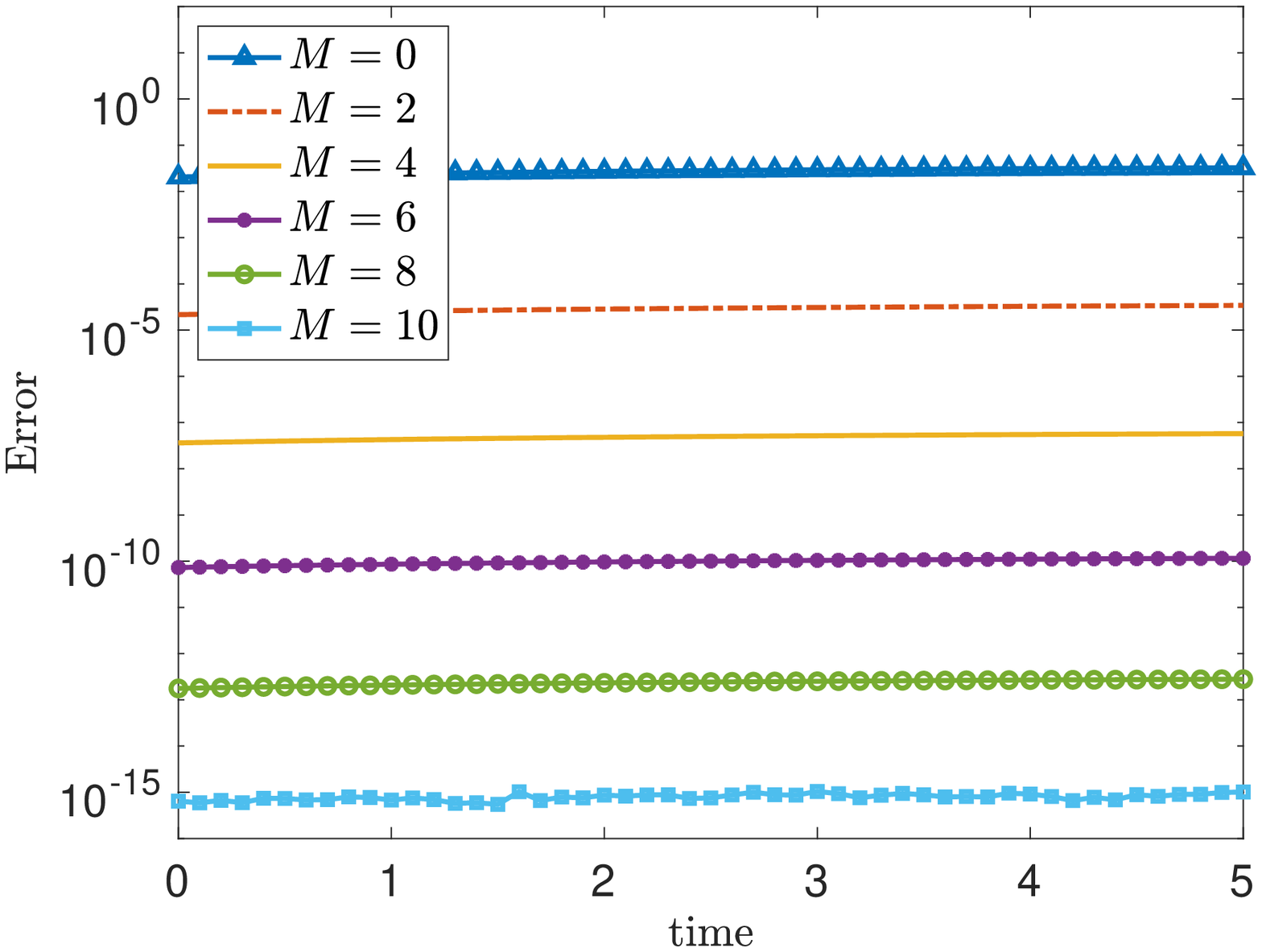}
\caption{\textbf{Test 1}. Left: $L^2$ error in the evaluation of $M4^M$ with respect to a reference solution at time $T = 5$. Right: evolution of the same error in the whole time interval $[0, 5]$. In both cases we considered $N = 10^6$ and $\Delta t = 0.1$ and the errors have been computed for two different $\kappa= 0.25$ and $\kappa = 0.75$ in the definition of the initial distribution \eqref{eq:f0_kac} .  }
\label{fig:kac3}
\end{figure}

\subsection{Test 2: 2D Maxwell model}\label{sect:2DMax}

We consider a 2D Boltzmann model \eqref{eq:boltzmann}-\eqref{eq:Q} with $B\equiv 1$ and binary interactions \eqref{eq:binary} where  
\[
\omega = \left( 
\begin{matrix}
\cos\theta \\
\sin\theta
\end{matrix}\right), \qquad \theta = 2\pi \xi, \qquad \xi \sim \mathcal U([0,1]). 
\]
We consider an uncertain initial distribution $v = (v_x,v_y)\in\mathbb R^2$
\begin{equation}
\label{eq:f0_max}
f_0(z,v) = \dfrac{\alpha^2(z) \mathbf v^2}{\pi} e^{-\alpha(z)\mathbf v^2 }, \qquad \mathbf v = \sqrt{v_x^2 + v_y^2}, 
\end{equation}
so that $f_0$ has the uncertain temperature 
\begin{equation}
\label{eq:T_2D}
T(z) = \dfrac{1}{\alpha(z)}. 
\end{equation}
An exact solution is given by (see Appendix \ref{app:2D})
\begin{equation}\label{eq:exact_max}
f(z,v,t) = \dfrac{1}{2\pi s(z,t)} \left[ 1 - \dfrac{1-\alpha(z)s(z,t)}{\alpha(z)s(z,t)} \left(1-\dfrac{\mathbf v^2}{2s(z,t)}  \right)\right]e^{-\frac{\mathbf v^2}{2s(z,t)}},
\end{equation}
where $s(z,t) = \dfrac{2 -e^{-t/8}}{2\alpha(z)}$. We will consider $\alpha(z) = 2+\kappa z$, with $z \sim \mathcal U([-1,1])$. 

The projection of the initial samples from the initial 2D distribution has been performed as described in Algorithm \ref{alg:samp0}.

In Figure \ref{fig:max1} and \ref{fig:max2} we report the contour plots of the expected solution and its variance at the initial time $t = 0$ and at time $t = 5$. In details, we compare the evolution of $\mathbb E[f](v,t)$ and $\textrm{Var}(f)(v,t)$ computed through the exact solution \eqref{eq:exact_max} and the DSMC-sG algorithm. A initial set of $N = 10^6$ particles distributed like $f_0(z,t)$ in \eqref{eq:f0_max} has been considered with a number of Galerkin projections $M = 5$.  In particular, for the reconstruction step we used a cartesian mesh in $[-L,L]^2$, $L=5$, composed by $N_v \times N_v = 100^2$ gridpoints. We observe again a very good agreement between the DSMC-sG algorithm and the exact solution of the kinetic model. 

\begin{figure}
\centering
\subfigure[$\mathbb E(f)$ Exact, $t=0$]{
\includegraphics[scale = 0.38]{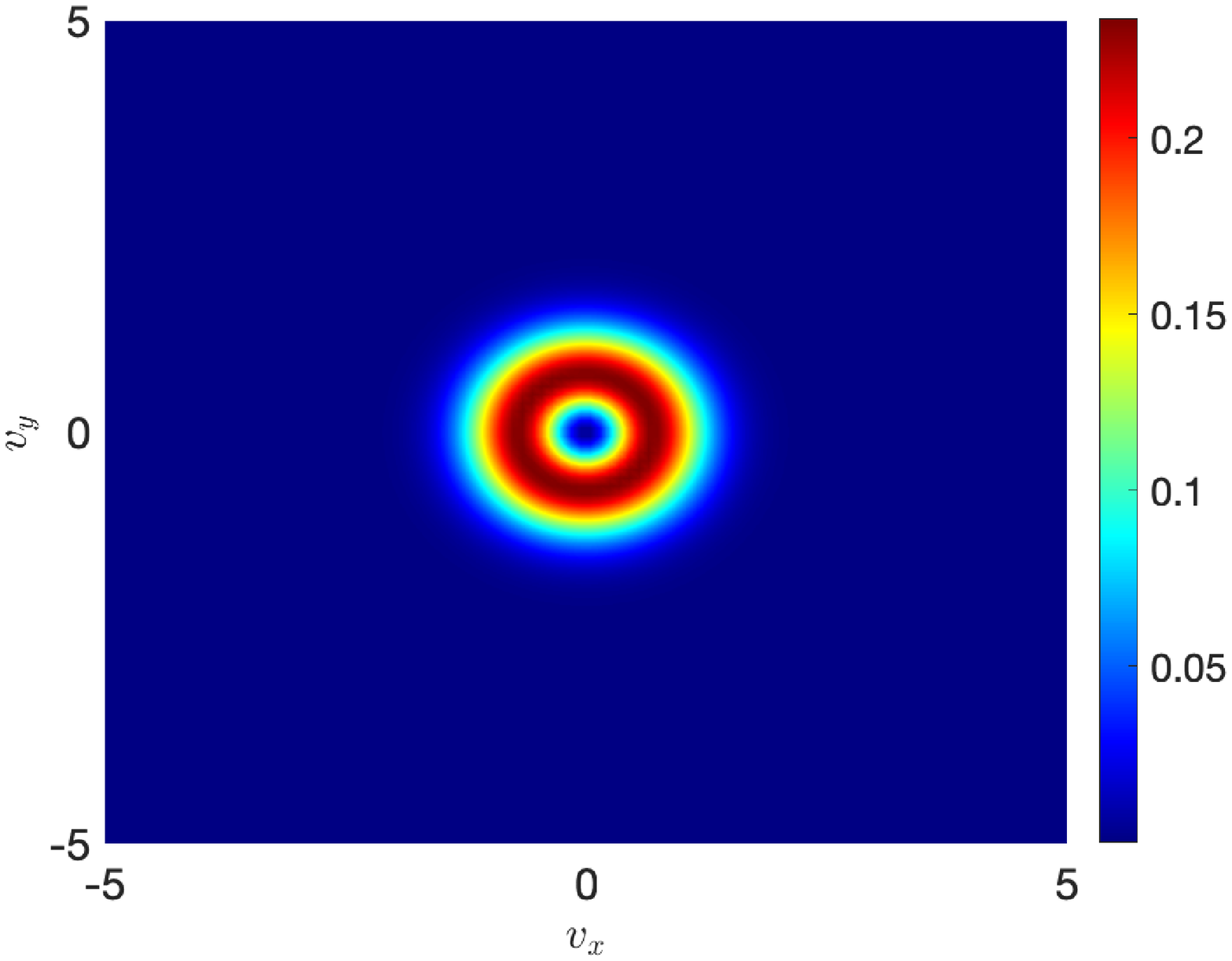}}
\subfigure[$\mathbb E(f)$ DSMC-sG, $t=0$]{
\includegraphics[scale = 0.38]{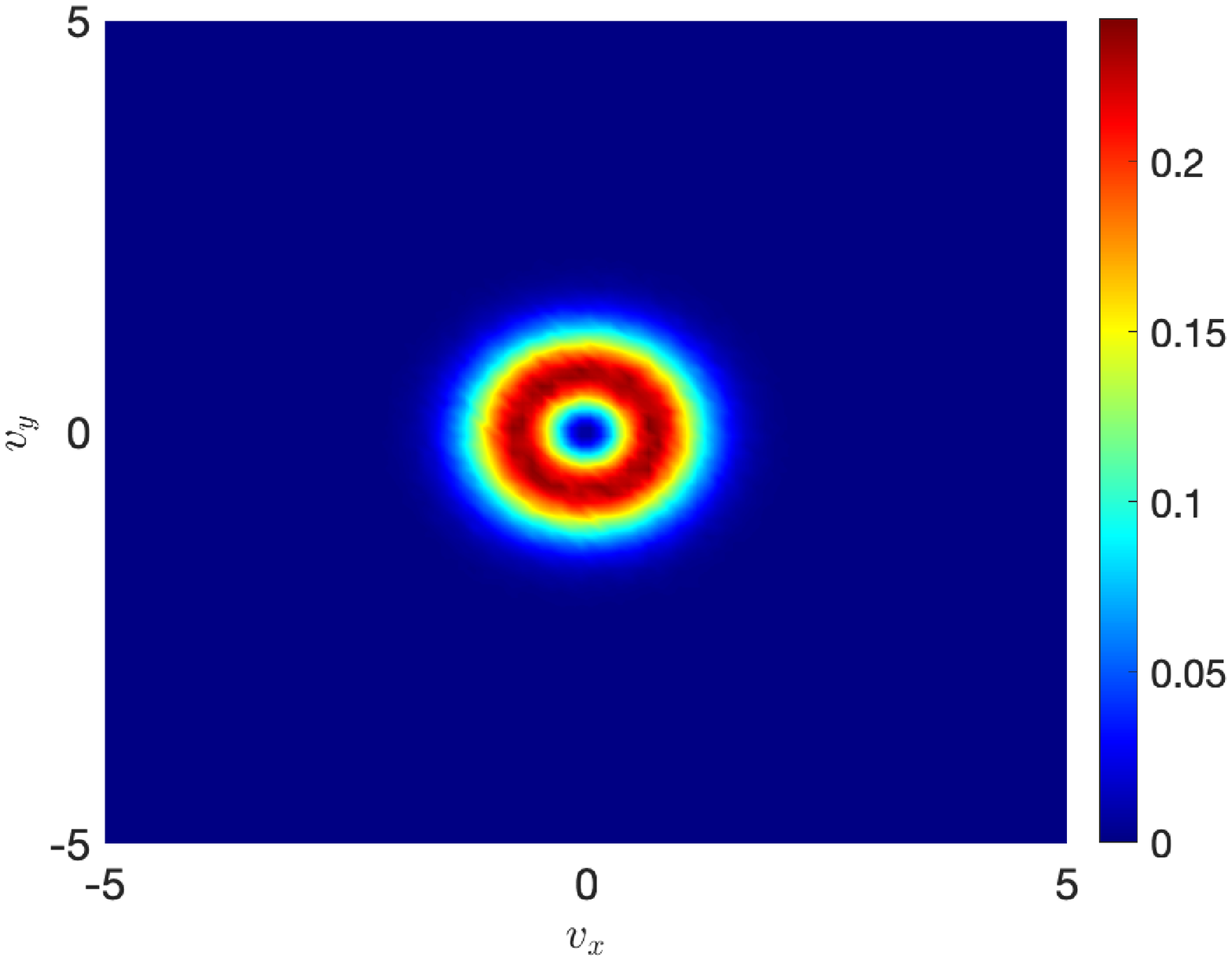}} \\
\subfigure[$\mathbb E(f)$ Exact, $t=5$]{
\includegraphics[scale = 0.38]{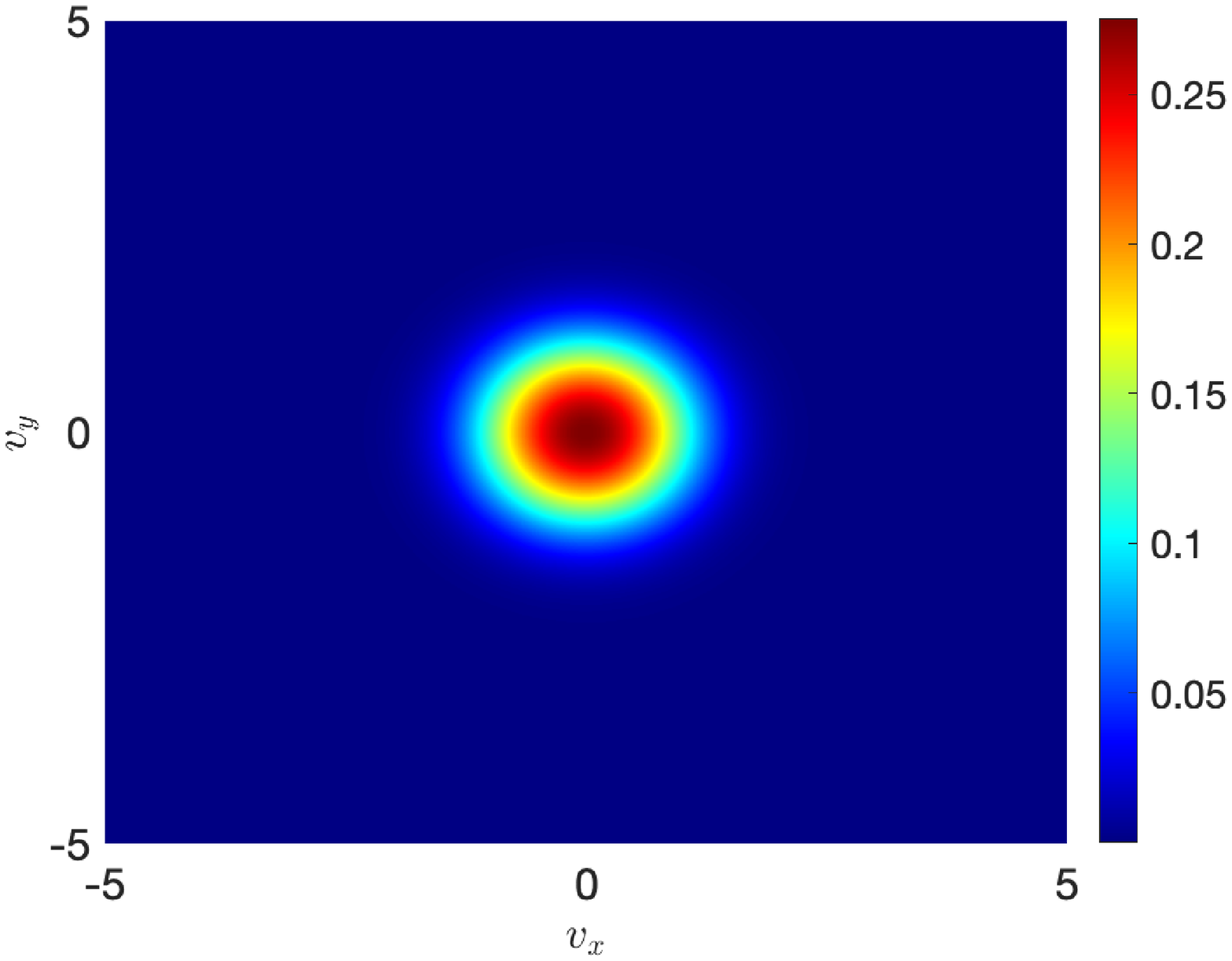}}
\subfigure[$\mathbb E(f)$ DSMC-sG, $t=5$]{
\includegraphics[scale = 0.38]{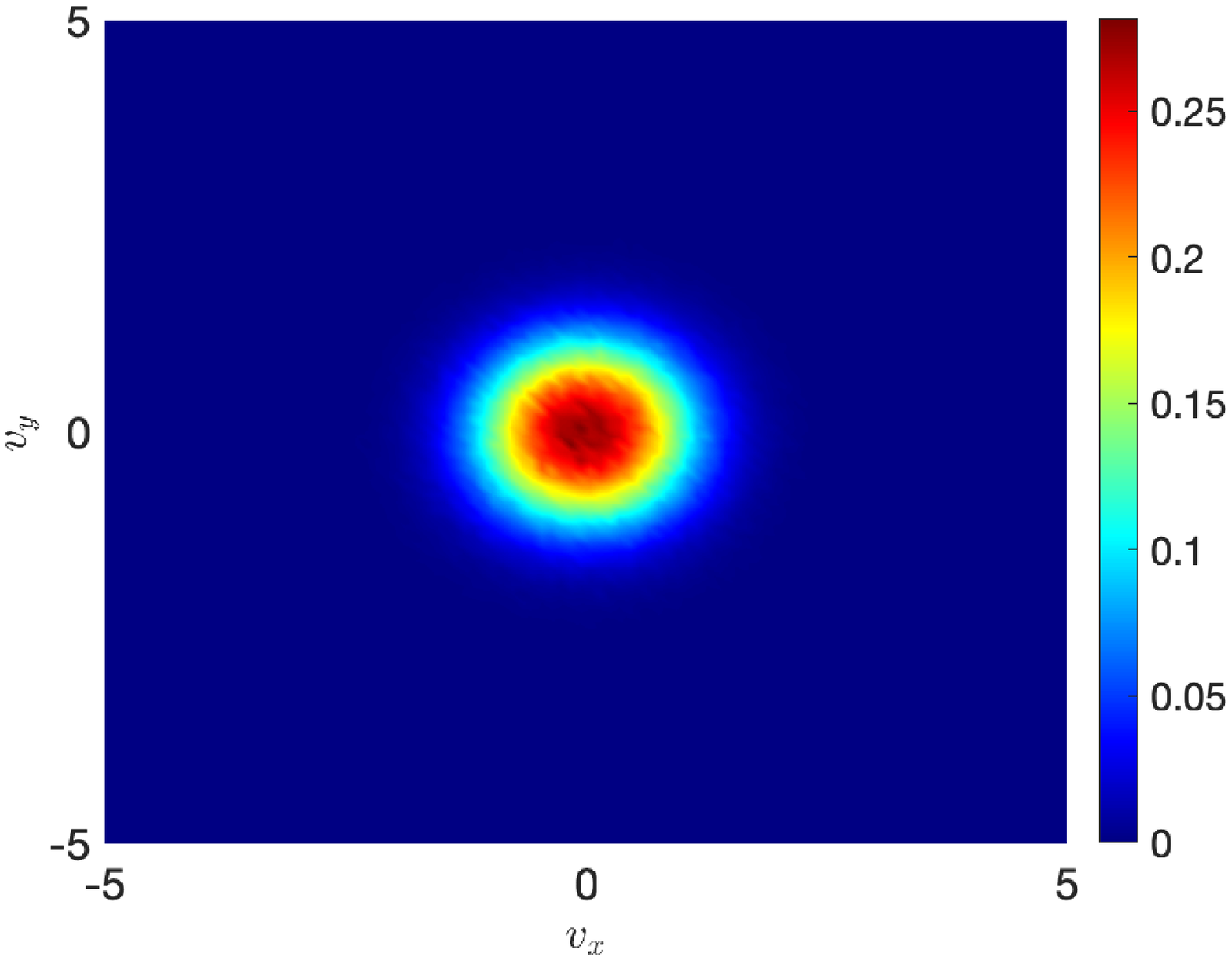}}
\caption{\textbf{Test 2}. Expected distribution of the 2D Boltzmann model for Maxwell molecules with uncertain temperature $1/\alpha(z)$ and $\alpha(z) = 2+\kappa z$, $z\sim \mathcal U([-1,1])$ and $\kappa = 0.25$. Left figure: expectation taken from the exact solution of the problem \eqref{eq:exact_max}. Right figure: reconstructed expectation through DSMC-sG method. We considered a set of $N = 10^6$ particles with $M = 5$ gPC expansion and $\Delta t = 10^{-1}$. }
\label{fig:max1}
\end{figure}

\begin{figure}
\centering
\subfigure[Var$(f)$ Exact, $t=0$]{
\includegraphics[scale = 0.38]{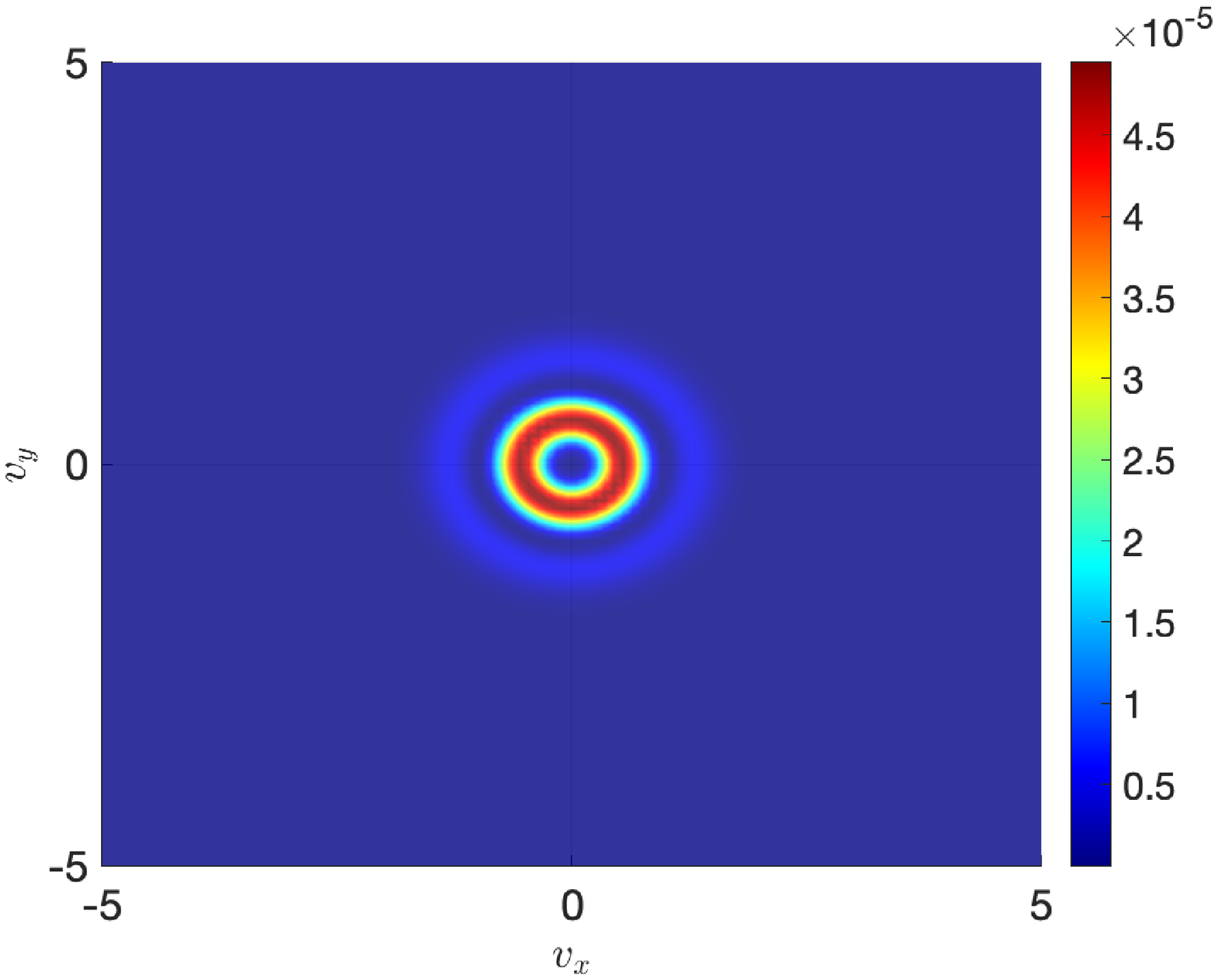}}
\subfigure[Var$(f)$ DSMC-sG, $t=0$]{
\includegraphics[scale = 0.38]{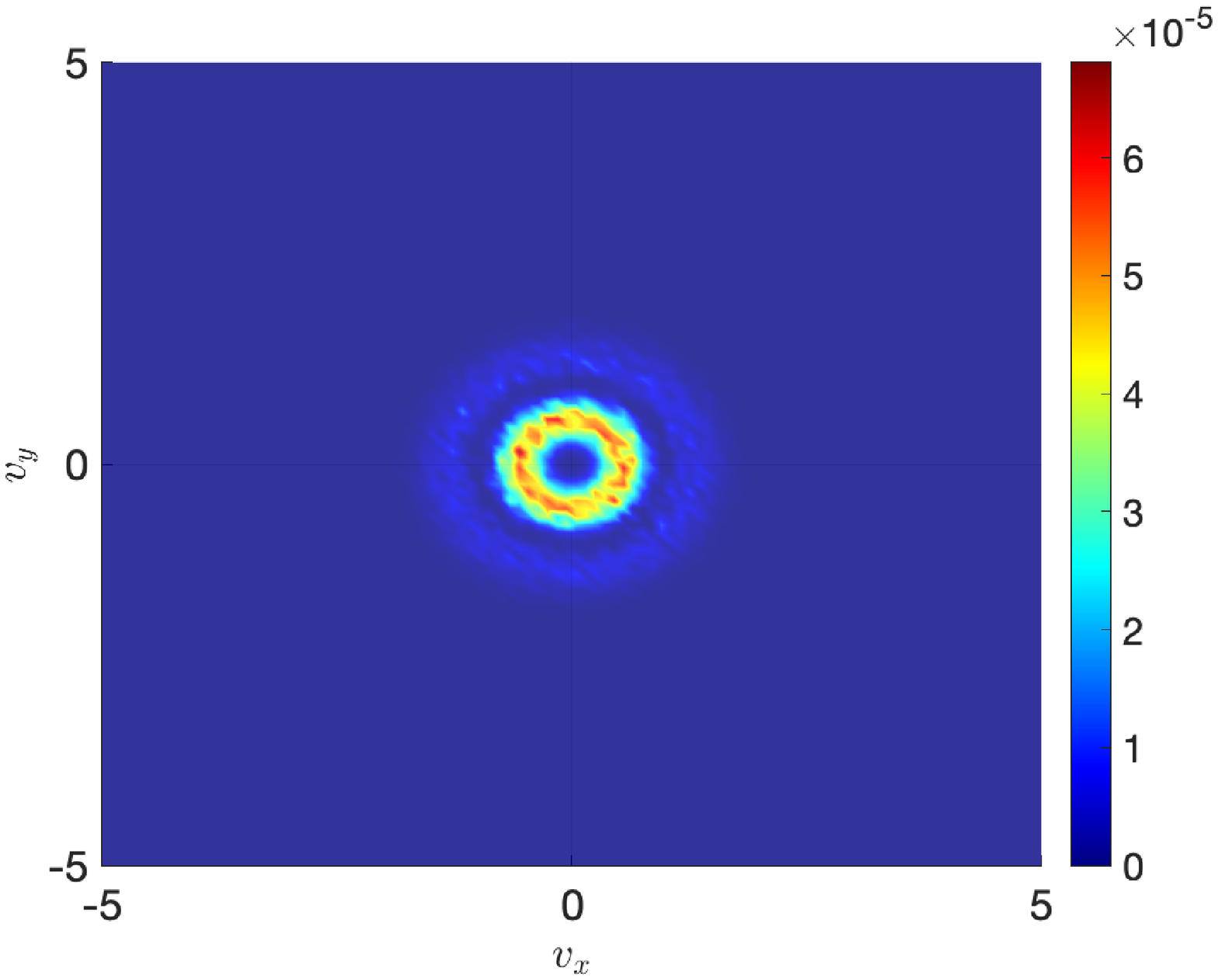}} \\
\subfigure[Var$(f)$ Exact, $t=5$]{
\includegraphics[scale = 0.38]{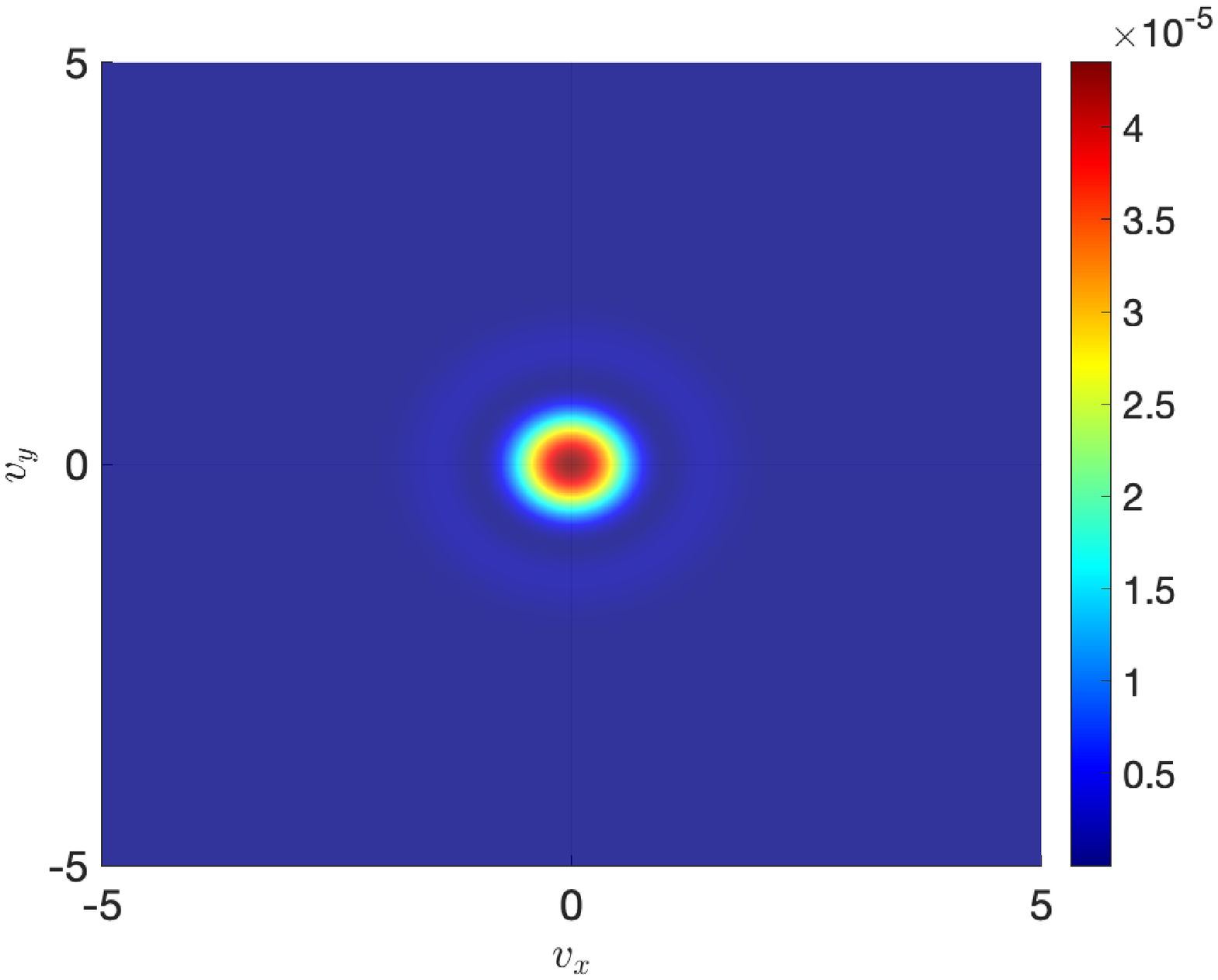}}
\subfigure[Var$(f)$ DSMC-sG, $t=5$]{
\includegraphics[scale = 0.38]{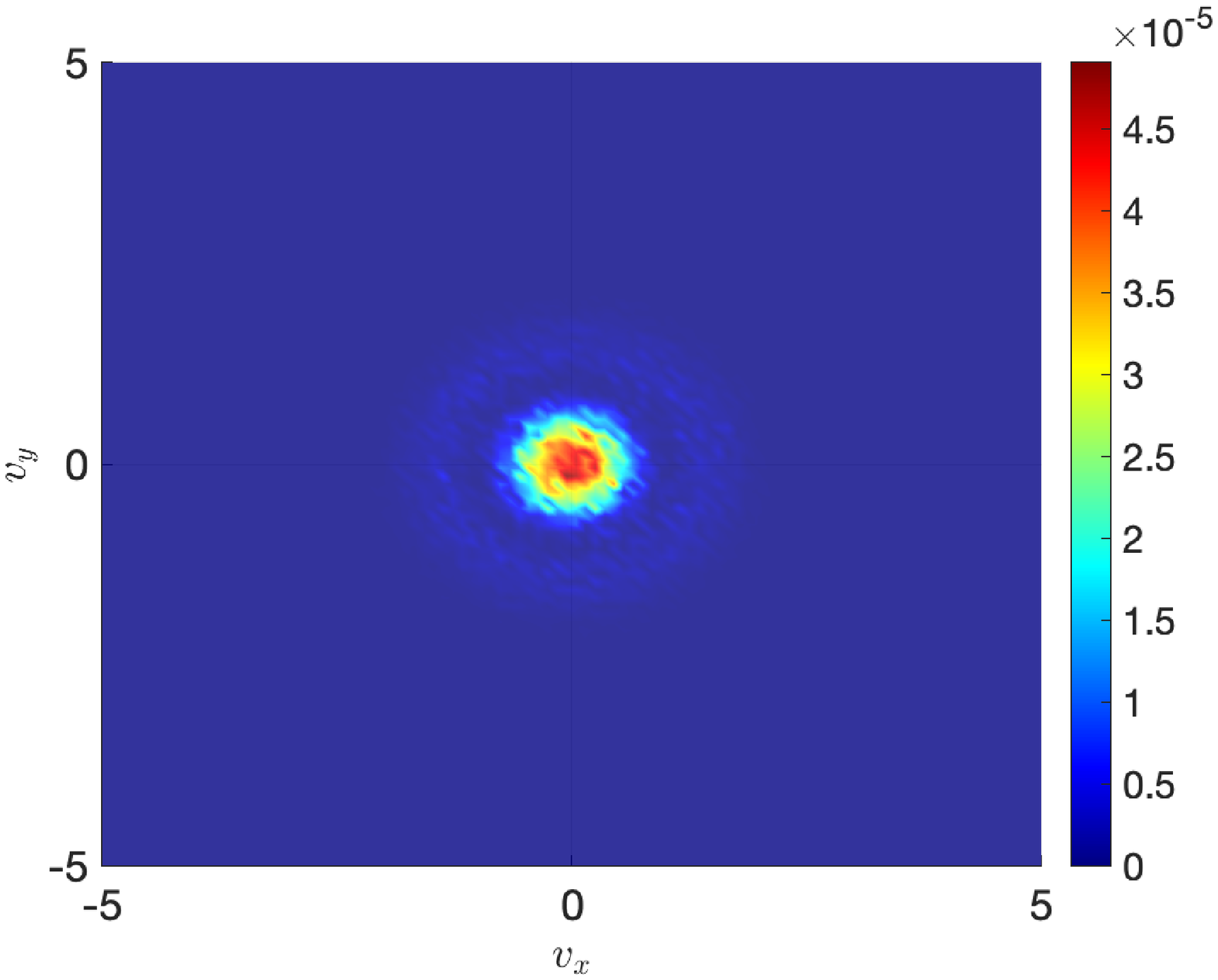}}
\caption{\textbf{Test 2}: Variance of the solution of the 2D Boltzmann model for Maxwell molecules with uncertain temperature $1/\alpha(z)$ and $\alpha(z) = 2+\kappa z$, $z\sim \mathcal U([-1,1])$ and $\kappa = 0.25$. Left figure: variance of the exact solution of the problem \eqref{eq:exact_max}. Right figure: reconstructed variance through DSMC-sG method. The rest of the parameters are the same as in Figure \ref{fig:max1}.}
\label{fig:max2}
\end{figure}

Furthermore, to emphasize the good agreement of the computed approximation for all times, we depict in Figure \ref{fig:max3} the evolution at times $t = 0,1,5$ of the marginal of $\mathbb E[f]$ and $\textrm{Var}(f)$. 

Finally, in Figure \ref{fig:max4} we present spectral convergence of the scheme computed through the fourth order moment of the 2D model with $\alpha(z) = 2+\kappa z$, $\kappa = 0.25$ and $\kappa = 0.7 5$ with $z\sim \mathcal U([-1,1])$. As reference solution we considered $M4$ at time $T = 5$ obtained with $N = 10^6$ particles and $M = 25$ Galerkin projections and the evolution is computed with $\Delta t = 10^{-1}$. In the right plot we present the decay of the $L^2(\Omega)$ error for increasing $M = 0,\dots,14$ in semilogarithmic scale. In the left plot we represent also the whole evolution of $M4$ computed through exact solution and through its DSMC-sG approximation. Similarly to the Kac model, we obtain numerical evidence of spectral convergence. 

\begin{figure}
\centering
\subfigure[$t = 0$]{
\includegraphics[scale = 0.25]{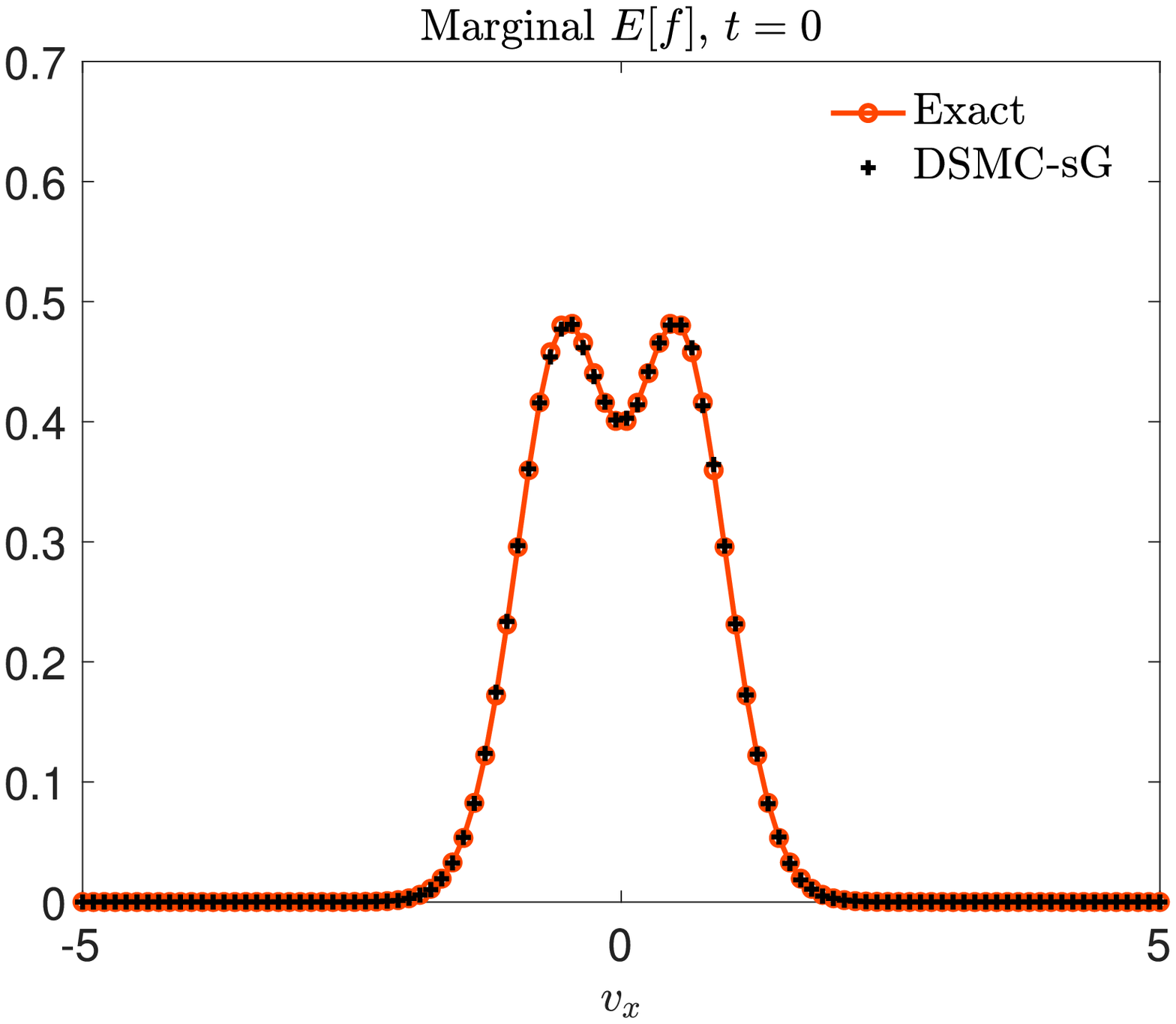}}
\subfigure[$t = 1$]{
\includegraphics[scale = 0.25]{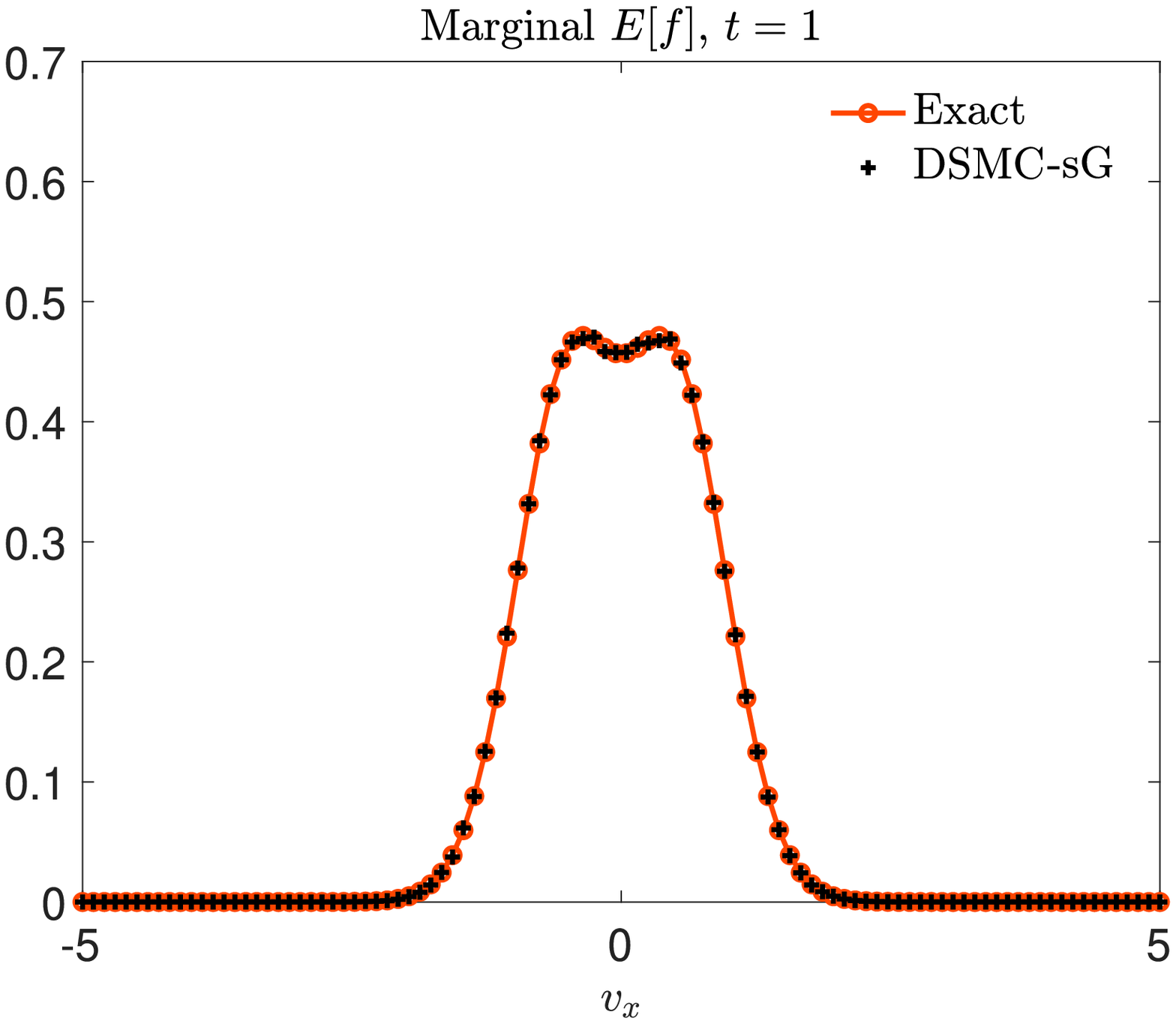}}
\subfigure[$t = 5$]{
\includegraphics[scale = 0.25]{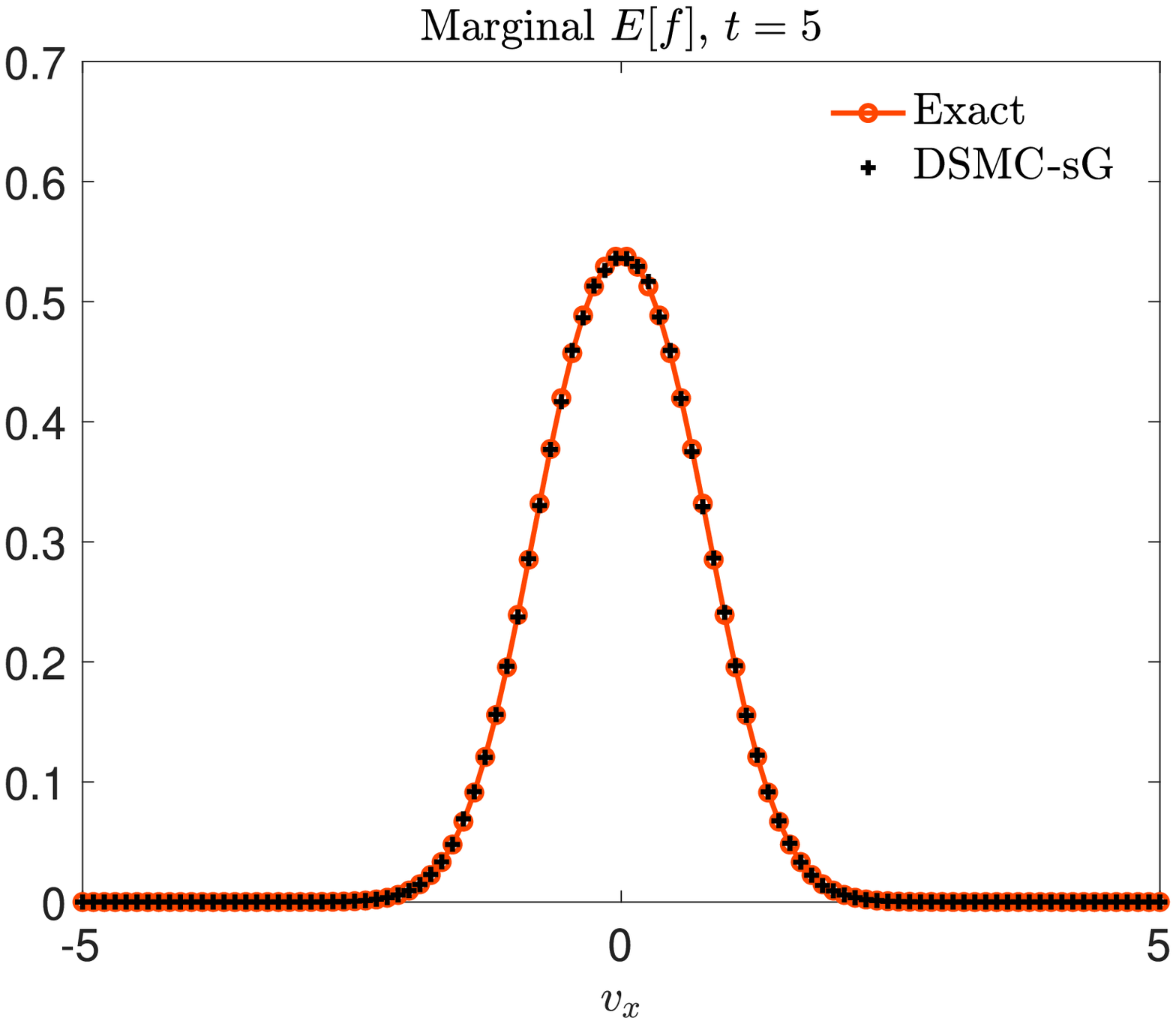}} \\
\subfigure[$t = 0$]{
\includegraphics[scale = 0.25]{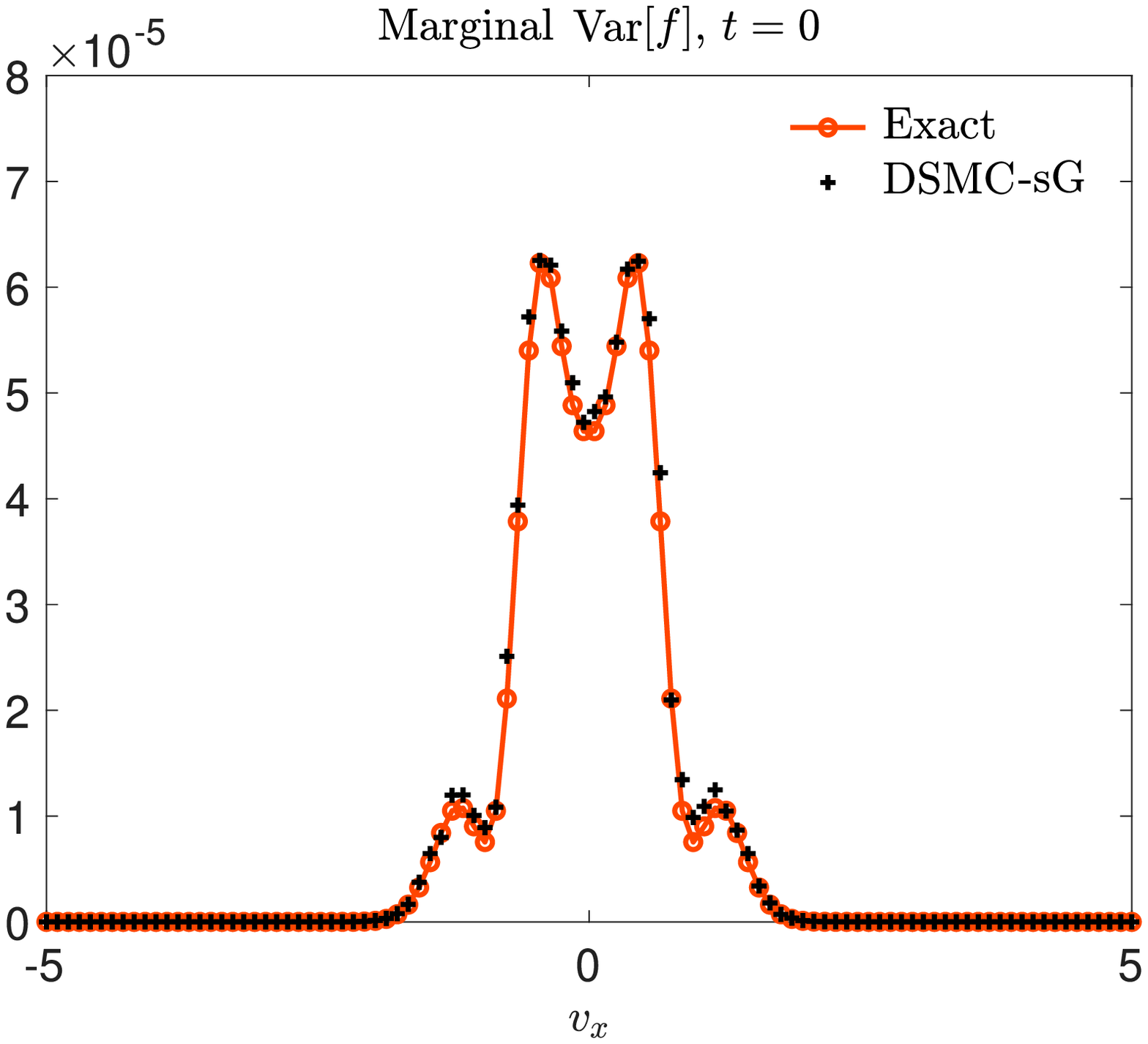}}
\subfigure[$t = 1$]{
\includegraphics[scale = 0.25]{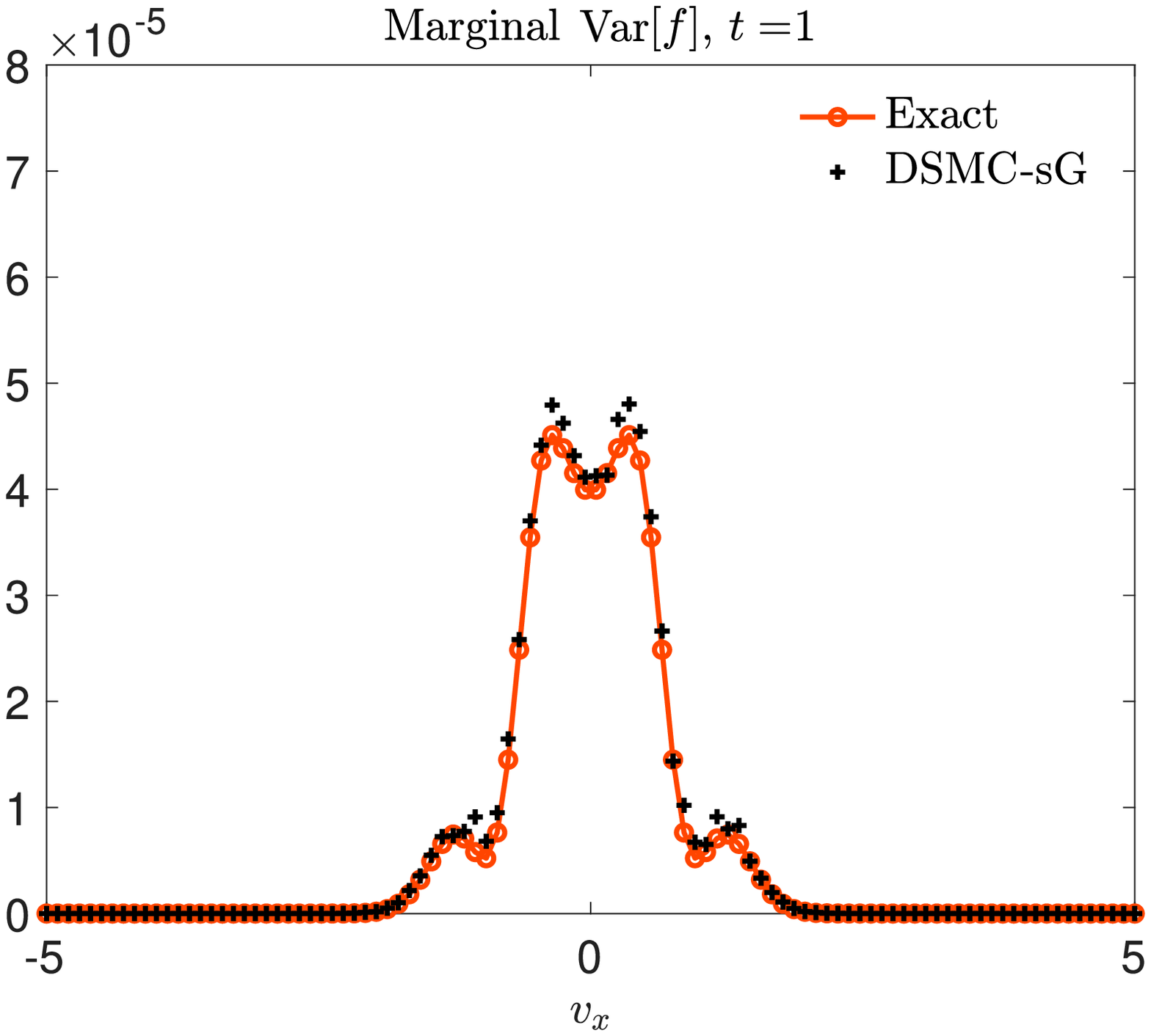}}
\subfigure[$t = 5$]{
\includegraphics[scale = 0.25]{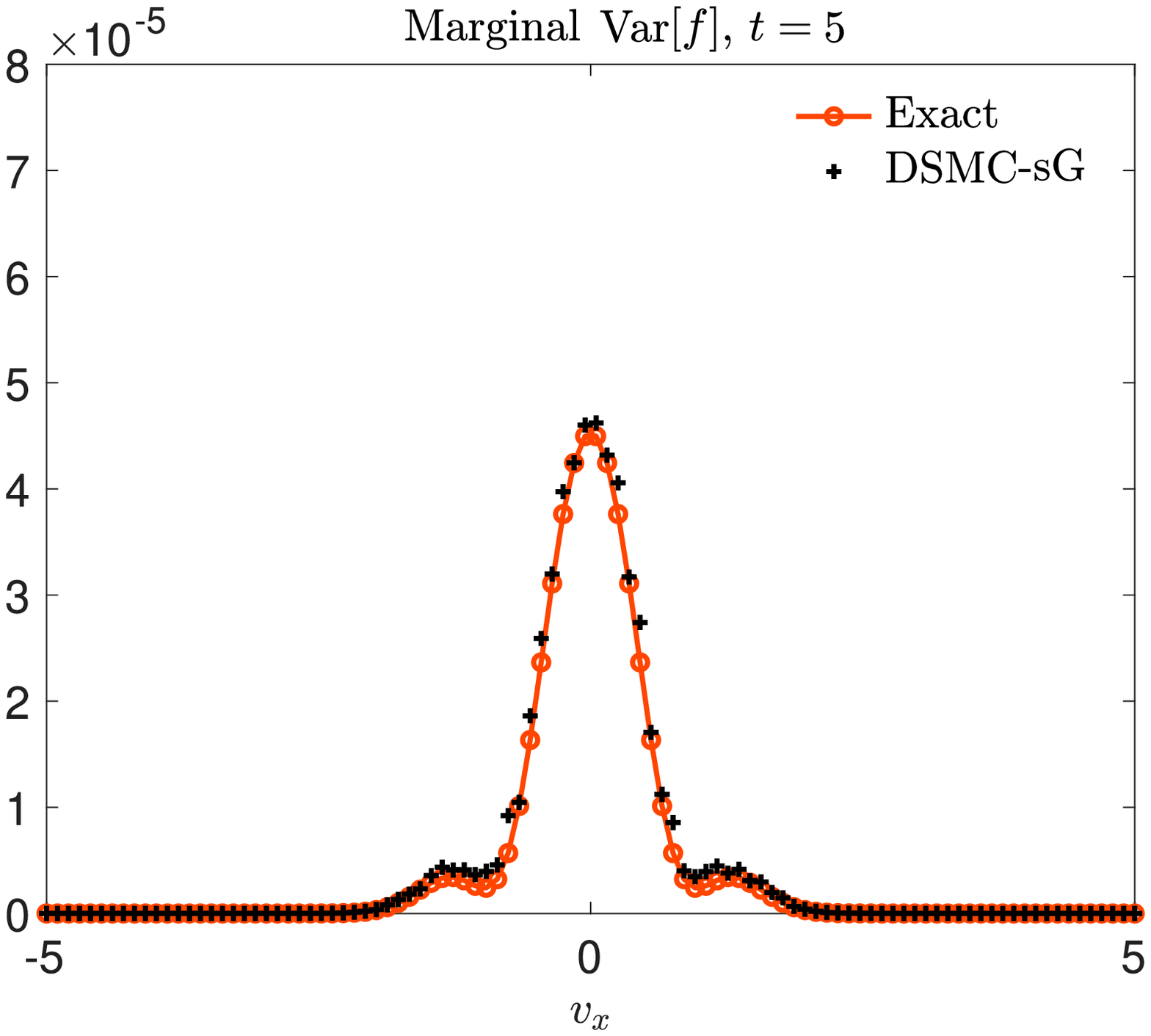}}
\caption{\textbf{Test 2}.  Evolution at times $t = 0,1,5$ of the marginal $\mathbb E[f]$ and $\textrm{Var}(f)$ from exact solution \eqref{eq:exact_max} and DSMC-sG approximation of the 2D Boltzmann model for Maxwell molecules with uncertain temperature. We considered $N = 10^6$ particles with $M= 5$ Galerkin projections and $\Delta t = 10^{-1}$. The reconstruction step has been performed in $[-5,5]^2$ through $100^2$ gridpoints. }
\label{fig:max3}
\end{figure}

\begin{figure}
\centering
\includegraphics[scale = 0.4]{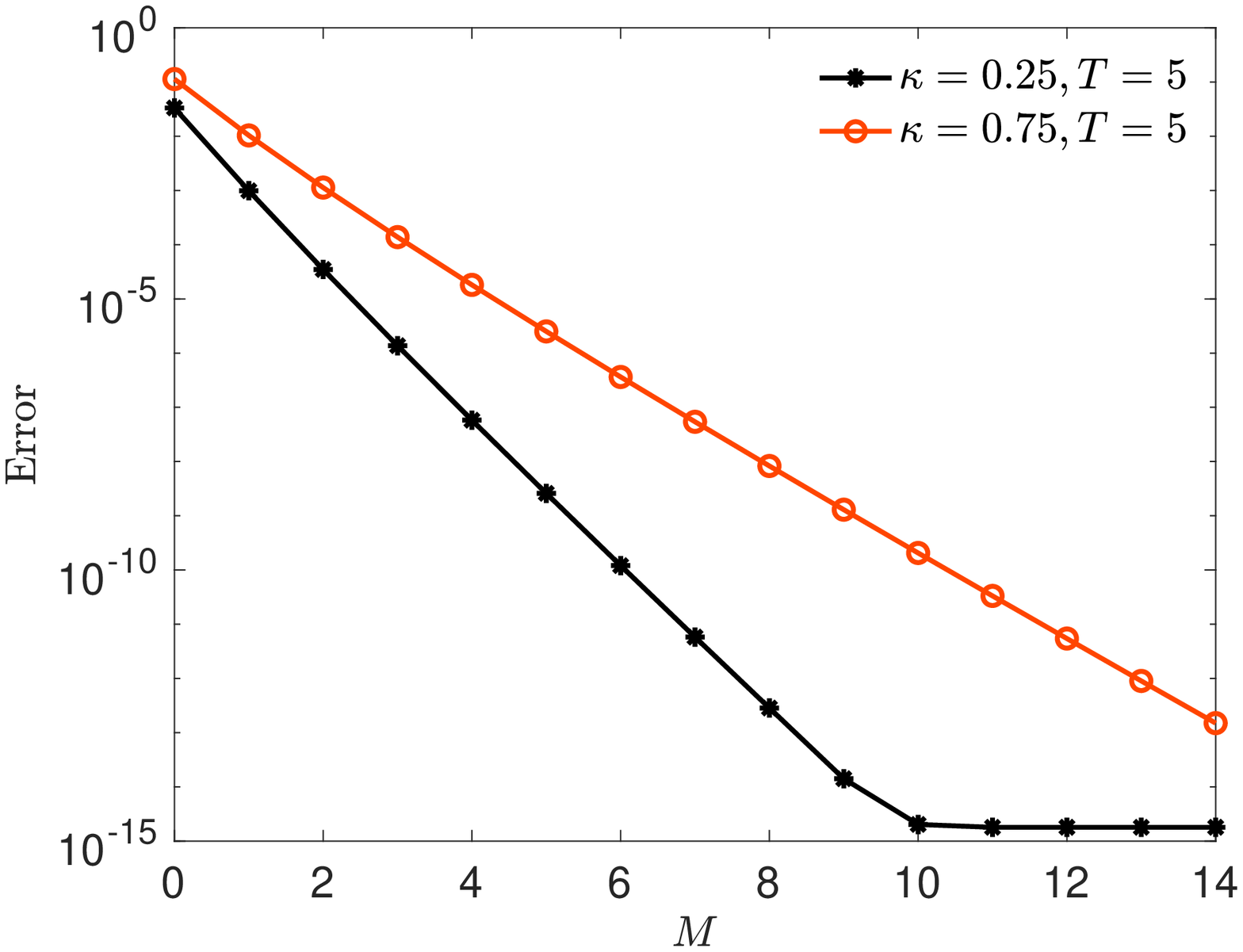}
\includegraphics[scale = 0.4]{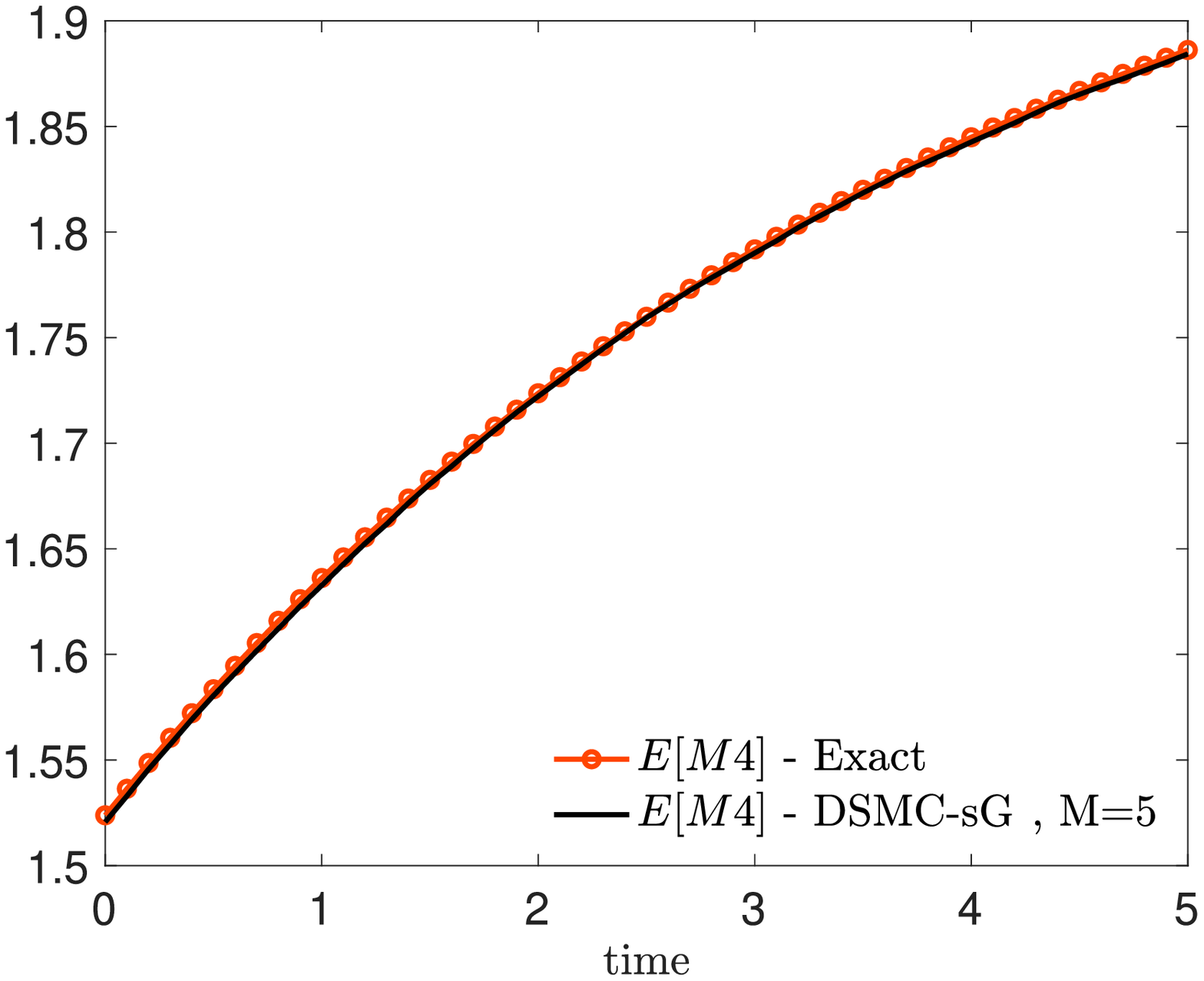}
\caption{\textbf{Test 2}. Left: Convergence of the $L^2(\Omega)$ error with respect to the fourth order moment obtained from a reference solution computed with $N = 10^6$ particles and $M = 25$ from the DSMC-sG methods. Right: evolution of the fourth order moment in the interval $[0,5]$ for exact and DSMC-sG approximation with $N = 10^6$ and $M = 5$.  }
\label{fig:max4}
\end{figure}

\subsection{Test 3: VHS molecules }\label{sect:test3}

We consider a 2D Boltzmann model with non-Maxwellian kernel of the form 
\[
B(z,|v(z)-v_*(z)|) = b_{\gamma}(\omega) |v(z)-v_*(z)|^\gamma,
\]
where $\gamma=(k-5)/(k-1)$. In particular we concentrate on the VHS model \cite{Bi76} where $b_\gamma(\theta) = C_\gamma$ and $C_\gamma$ is a positive constant. The case $\gamma=0$ refers to the model for Maxwellian molecules whereas $\gamma = 1$ describes a hard sphere gas. Let us consider an uncertain initial distribution function which is a sum of two Gaussian distributions with uncertain variance
\begin{equation}
\label{eq:f0_VHS}
f(z,v,0) = \dfrac{1}{2\pi\sigma^2(z)} \left[ e^{\frac{-|v-2\sigma(z) e_1|^2}{2\sigma^2(z)}} + e^{\frac{-|v+2\sigma(z) e_1|^2}{2\sigma^2(z)}}  \right],
\end{equation}
where $e_1 = (1,0)$. This distribution can be employed to check the evolution of the components of the stress tensor defined as
\begin{equation}\label{eq:Pij}
P_{ij}(z,t) = \int_{\mathbb R^2} (v_i-u_i)(v_j-u_j)f(z,v,t)dv,\qquad i,j = 1,2,
\end{equation}
where $u_i$ are the components of the mean velocity. For Maxwellian molecules, i.e. $\alpha = 0$, we can derive exact evolution of the components of the stress tensor
\begin{equation}
\label{eq:P_stress}
P_{11}(z,t) = T(z) + \dfrac{1}{2}w(t), \qquad P_{22}(z,t) = T(z) - \dfrac{1}{2}w(t),
\end{equation}
being 
\begin{equation}
T(z) = \sigma^2(z), 
\label{eq:TE2D2}
\end{equation} 
the temperature and $w(t) = w_0 e^{-t/2}$ and $w_0 = 4\pi$, see \cite{BR,Par1}. In the following tests we will consider 
\begin{equation}
\label{eq:sigma_VHS}
\sigma(z) = \dfrac{\lambda\pi}{6} \left( 1 + \kappa z \right), \qquad z\sim \mathcal U([-1,1]), 
\end{equation}
with $\lambda = \frac{2}{3+\sqrt{2}}$. In Figure \ref{fig:test3_1} we compute the evolution of the components of the expected stress tensor in the case $\gamma = 0$ and for the values $\kappa = 0.1$, and $\kappa = 0.5$. The evolution obtained with DSMC-sG method, with $N= 10^6$ particles and  Galerkin projections of order $M = 5$, is then compared with the exact one. The expected values are accurately approximated by the scheme. Furthermore, we computed  $\textrm{Var}(P_{11})(t)$ and $\textrm{Var}(P_{22})(t)$ to build the variability area highlighted in grey in Figure \ref{fig:test3_1} through the standard deviations from the expected stress tensor.

\begin{figure}
\subfigure[$\kappa = 0.1$]{
\includegraphics[scale = 0.4]{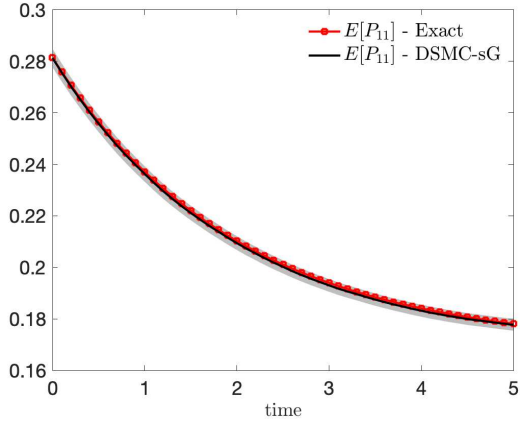}}
\subfigure[$\kappa = 0.1$]{
\includegraphics[scale = 0.4]{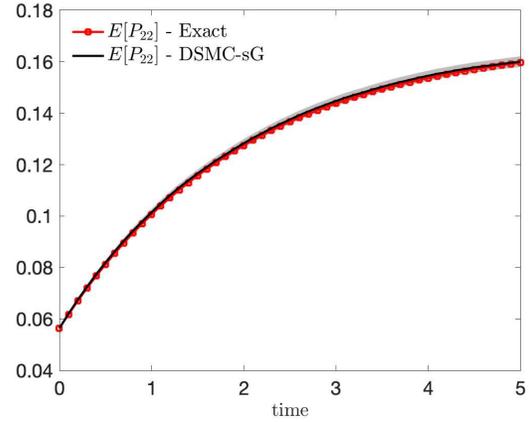}}\\
\subfigure[$\kappa = 0.5$]{
\includegraphics[scale = 0.4]{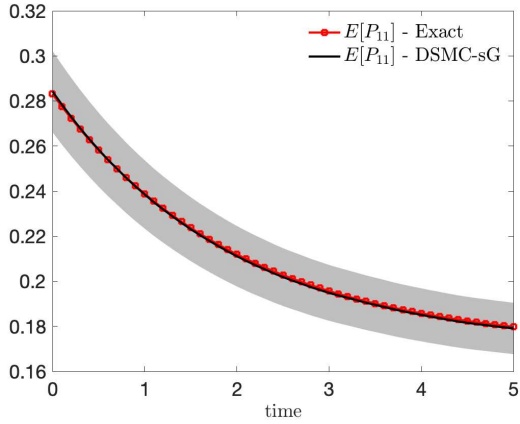}}
\subfigure[$\kappa = 0.5$]{
\includegraphics[scale = 0.4]{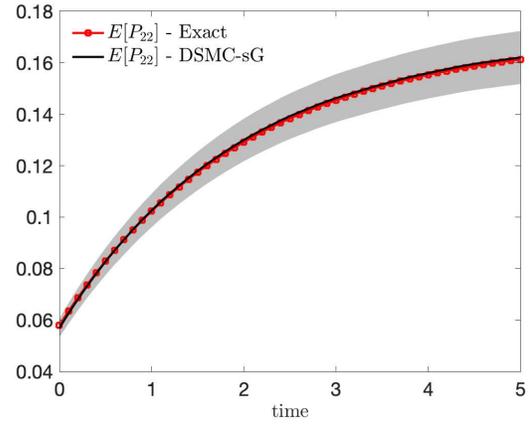}}
\caption{\textbf{Test 3}. Evolution of the expected $P_{11}(z,t)$ from initial density \eqref{eq:f0_VHS} in the case $\gamma = 0$ and $C_0 = 1/2\pi$ and two level of initial uncertainty $\kappa = 0.1$ and $\kappa = 0.5$ in \eqref{eq:sigma_VHS}. We compare the exact evolution \eqref{eq:P_stress} with the one obtained from DSMC-sG scheme for VHS molecules with $\gamma = 0$ with $N = 10^6$ particles with $M = 5$ Galerkin projections. In grey we highlighted the displacement obtained through standard deviations }
\label{fig:test3_1}
\end{figure}

The evolution of the expected stress tensor for $\gamma = 0,1,2$ is shown in Figure \ref{fig:test3_2}. In agreement with the deterministic case, see \cite{Par1} we may observe how the decay of $\mathbb E[P_{11}]$ is stronger for $\gamma >0$ than the Maxwellian case, corresponding to $\gamma = 0$. This behavior is emphasized in semilogarithmic scale in the left picture of Figure \ref{fig:test3_2}.  Each dynamics is obtained through DSMC-sG scheme with $N = 10^6$ particles and $M=5$ Galerkin projections of the binary collisions.  

\begin{figure}
\includegraphics[scale  = 0.40]{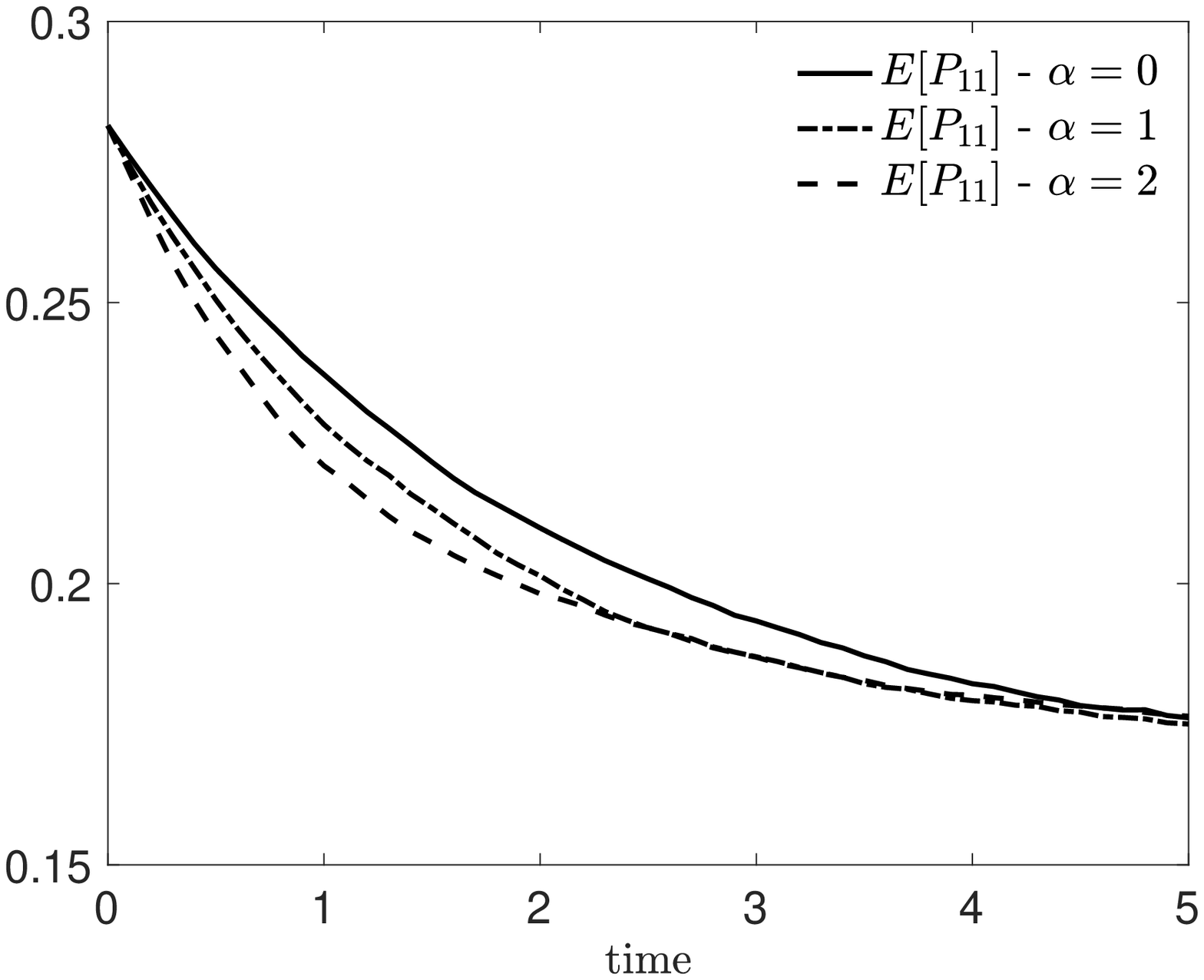}
\includegraphics[scale  = 0.40]{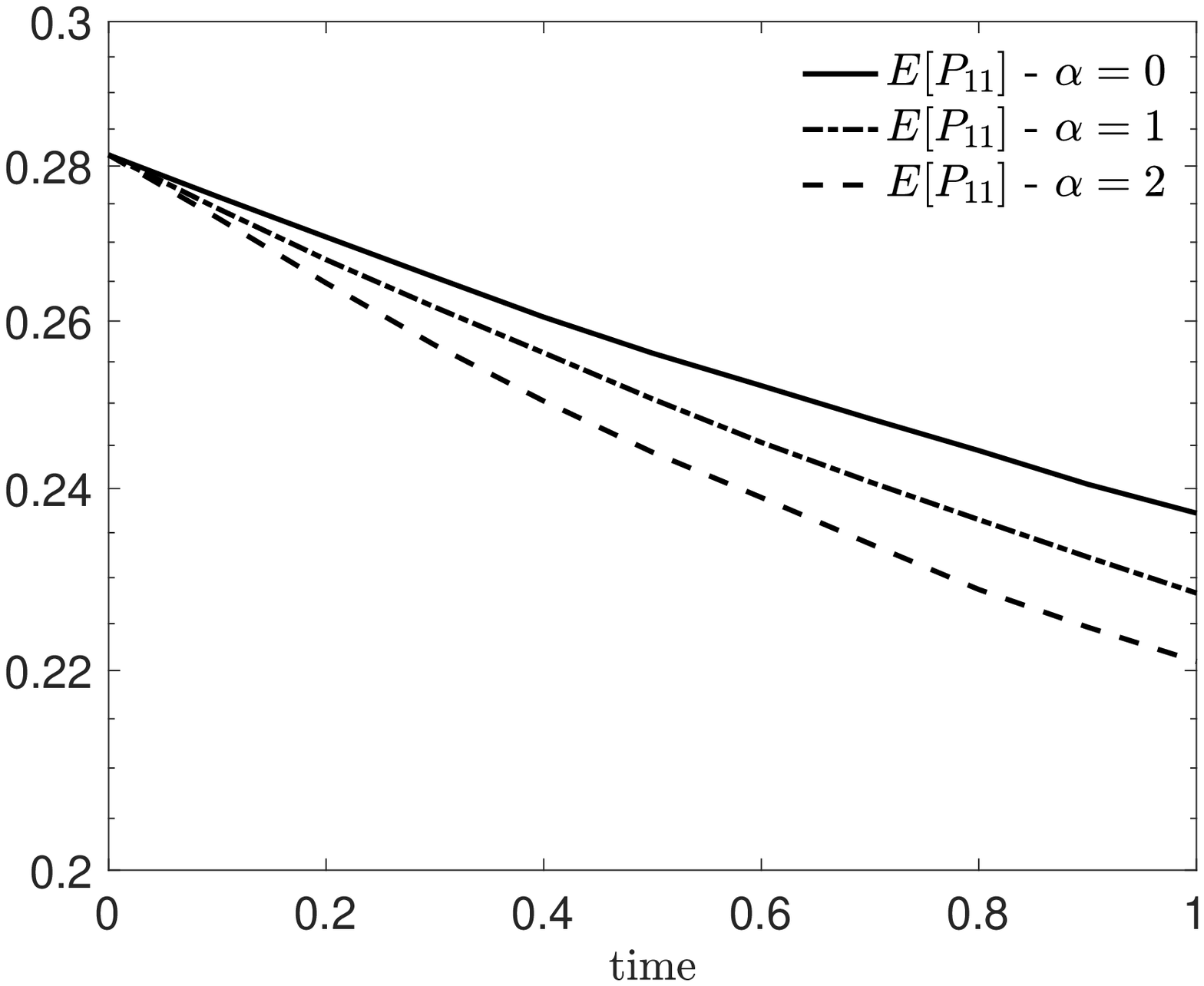}
\caption{\textbf{Test 3}.  Evolution of the expected stress tensor $\mathbb E[P_{11}]$ for $\gamma = 0,1,2$ obtained from DSMC-sG scheme with $N = 10^6$ and $M= 5$, $\Delta t = 10^{-1}$. In the right picture we present the evolution in the interval $[0,1]$ in semilogarithmic scale to highlight the different expected trends in hard sphere gases.   }
\label{fig:test3_2}
\end{figure}

In the reformulation \eqref{eq:binary_sig}-\eqref{eq:binary_terma} of the binary collision term for hard gases in Section \ref{sect:VHS}, in order to recover spectral accuracy, we proposed to replace the discontinuous function  present in the original dynamics \eqref{eq:binary_VHS} with a smooth sigmoid function $K(\beta(\cdot))$ coupled with a post-interaction thermalization process. The function needs to be an approximation of the indicator function $\Psi(\cdot)$ and has been introduced in order to preserve the smoothness of projected quantities required in a stochastic Galerkin approach. Let us consider the sigmoid function
\[
\begin{split}
K(\beta(x-y)) = \dfrac{ \textrm{tanh}(\beta(x-y))+1 }{2}, \qquad \beta >0.
\end{split}\] 
In Figure \ref{fig:test3_3} we compute the $L^2$ error of the stochastic Galerkin methods for increasing number of projections with respect to the approximated stress tensor $P_{11}^M$. A reference solution computed with $M = 50$ and $N = 10^6$ particles is considered from the initial distribution \eqref{eq:f0_VHS} with temperature \eqref{eq:TE2D2} and $\kappa = 0.1$. If we consider the original binary dynamics \eqref{eq:binary_VHS}, even if the expectation of $P_{11}(z,t)$ is well described, it can be seen that spectral accuracy of the method is lost due to the discontinuity of function $\Psi$. On the other hand, the same test performed for the modified binary dynamics \eqref{eq:binary_sig} without thermalization (see Figure \ref{fig:test3_3}, left) recovers spectral accuracy but at the price of a dissipative dynamics since the regularized interaction is no more conservative for the energy. For increasing $\beta\gg 0$, as expected, the convergence of the scheme deteriorates even if energy dissipation vanish and the expectation of $P_{11}(z,t)$ is well approximated. Coupling now, \eqref{eq:binary_sig} with the thermalization process \eqref{eq:binary_terma} (see Figure \ref{fig:test3_3}, right) we recover an accurate evolution of $P_{11}(z,t)$ together with the spectral convergence of the scheme for moderate values of $\beta>0$, which, as expected, deteriorates for $\beta\gg 0$. 

\begin{figure}
\centering
\includegraphics[scale  = 0.40]{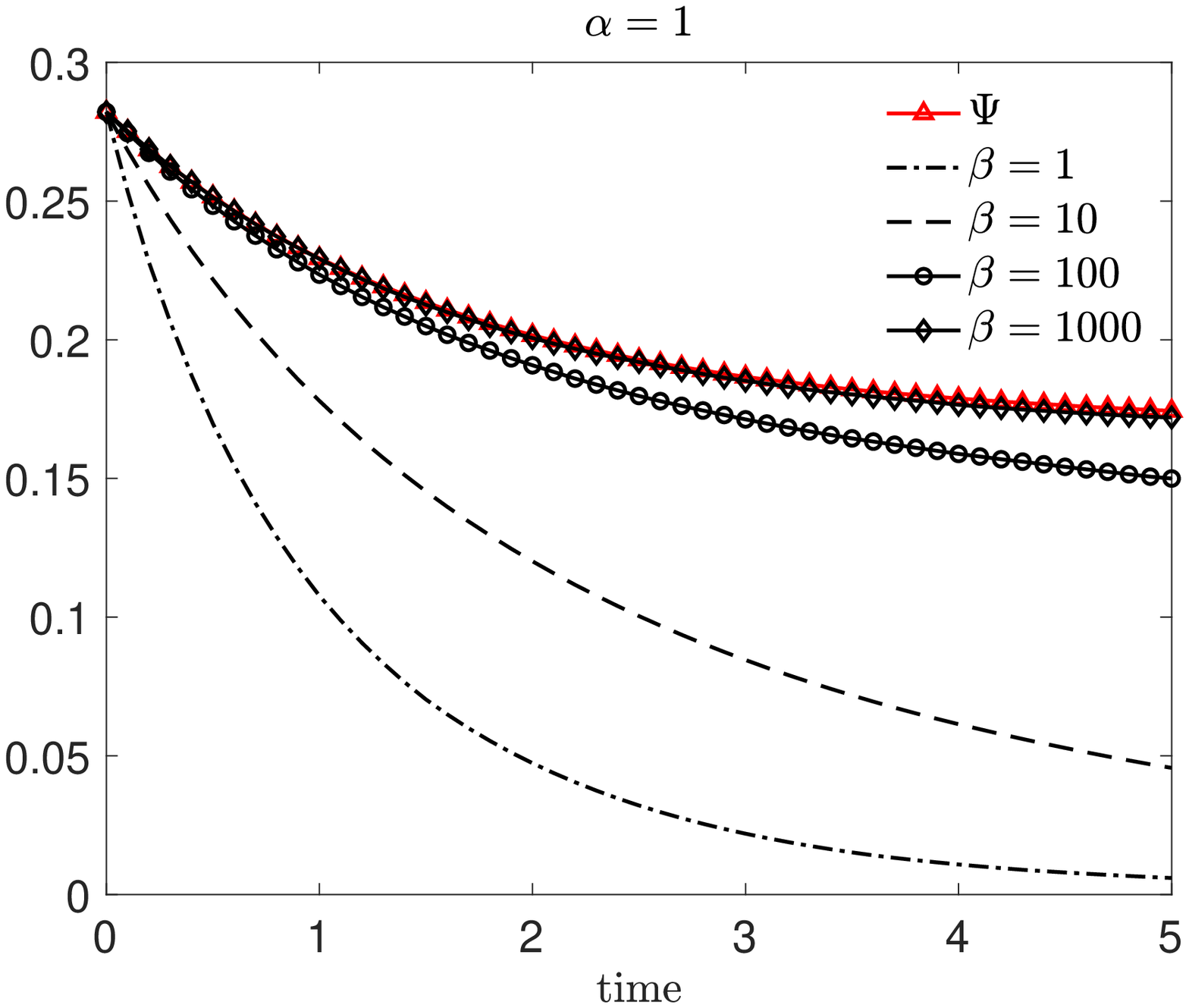}
\includegraphics[scale  = 0.40]{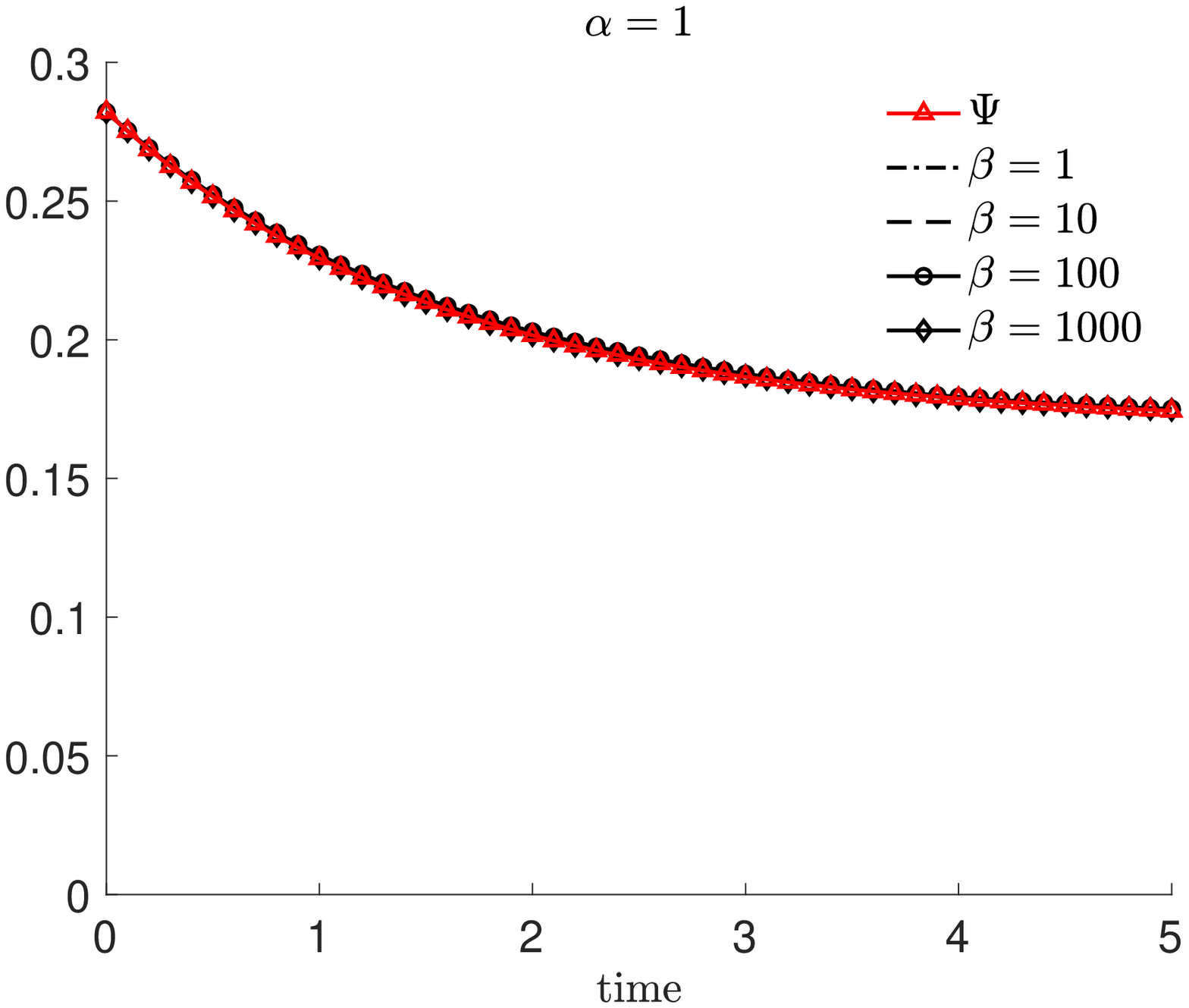} \\
\includegraphics[scale  = 0.40]{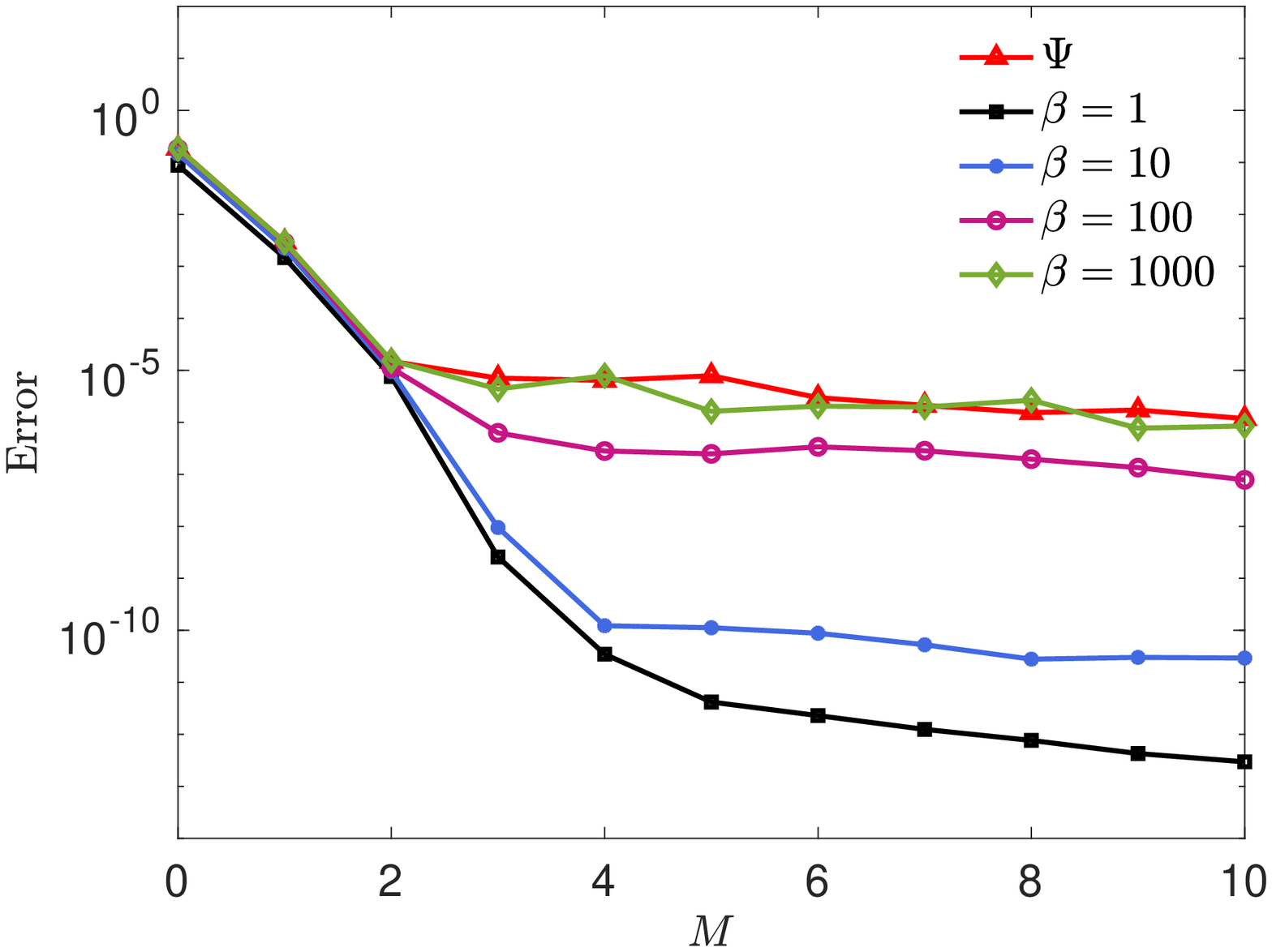}
\includegraphics[scale  = 0.40]{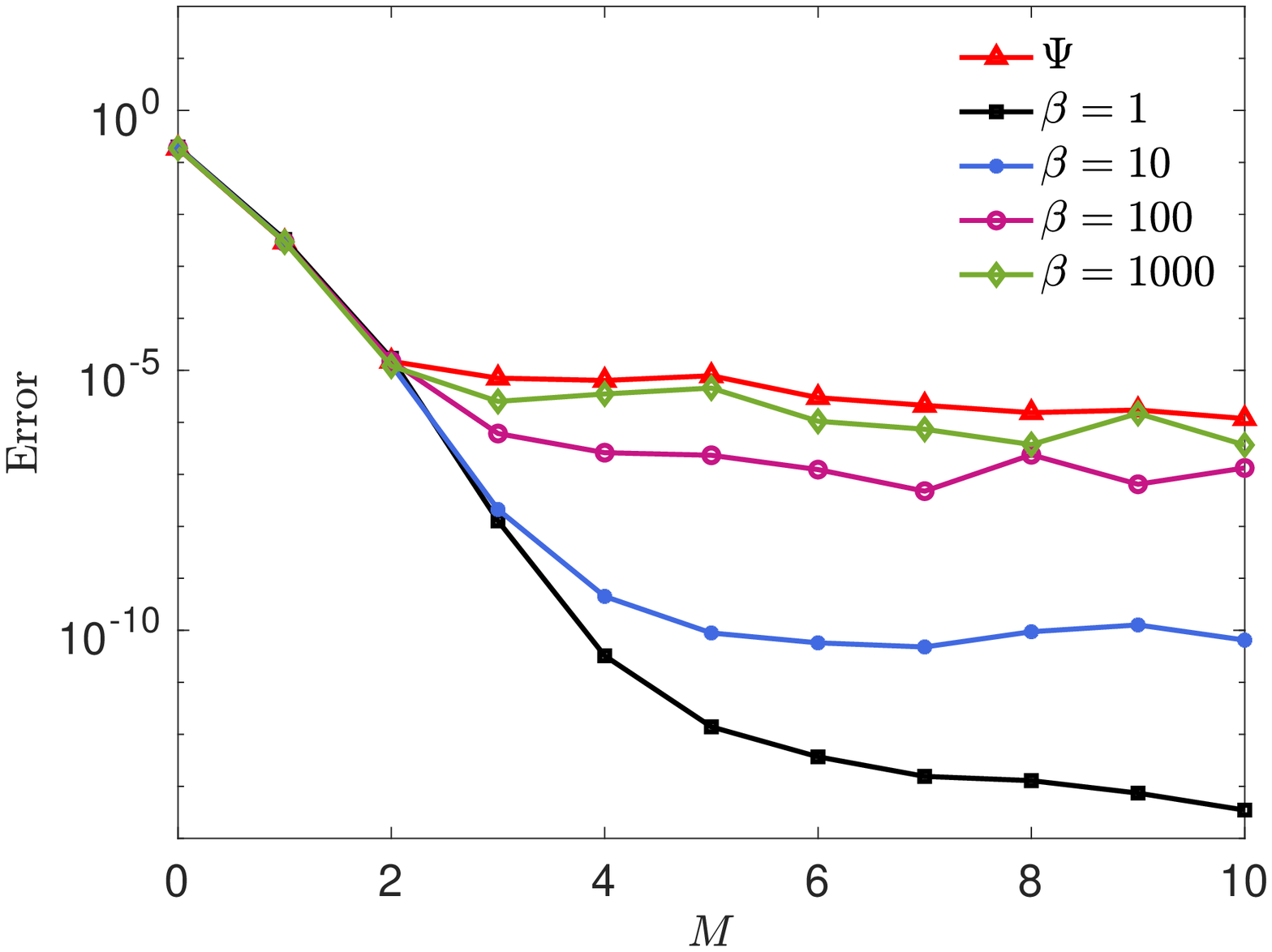}
\caption{\textbf{Test 3}. Convergence of the $L^2$ error of the DSMC-sG scheme for the VHS model with $\gamma = 1$  where the binary dynamics are given by \eqref{eq:binary_VHS} (left column) or \eqref{eq:binary_sig}-\eqref{eq:binary_terma} (right column). The error is computed from $P^M_{11}(z,t)$ at time  $t = 1$, $\Delta t = 10^{-1}$, with $N = 10^6$ particles. Reference solution computed with $M = 50$.   }
\label{fig:test3_3}
\end{figure}

\subsection{Test 4: VHS molecules with bivariate uncertainty}
In the last test we consider a 2D Boltzmann VHS model with bivariate uncertainty $z = (z_1,z_2)$ with independent components and with the same distribution $ p(z)$.  In details, we consider an uncertain initial distribution function of the form \eqref{eq:f0_VHS} with $\sigma(z_1) = \frac{\lambda\pi}{6} (1+\kappa_1 z_1)$, $\kappa_1>0$. Furthermore, we consider an uncertain interaction kernel of the form
\begin{equation}\label{eq:BVHS}
B(z_2,|v(z_1,z_2) - v_*(z_1,z_2)|) = C_0 |v(z_1,z_2) - v_*(z_1,z_2)|^{\gamma(z_2)}, 
\end{equation}
with $C_0 = 1/2\pi$. In this test we will consider $z_1, z_2 \sim \mathcal U([-1,1])$, and $\gamma(z_2) = \kappa_2(1+z_2)$, $\kappa_2>0$. These choices are coherent with the case where both the initial temperature of the gas and the nature of the molecules are affected by uncertainty. In particular,  we consider a collision kernel of the form \eqref{eq:BVHS} characterizing gases with collisions that may span from Maxwellian to hard potentials. 

In Figure \ref{fig:VHS2D2D} we report the evolution of the expected values of the  diagonal components of the stress tensor $\mathbb E[P_{11}]$,  $\mathbb E[P_{22}]$ and of their variance in the bivariate case with $\kappa_1 = 0.5, \kappa_2 = 1$ and in the univariate case obtained with $\kappa_1 = 0.5, \kappa_2 = 0$. The univariate setting is here conformal with the Maxwellian case studied in Section \ref{sect:test3}. It can be observed that the uncertainty on the collision kernel determines a faster trend to equilibrium in the resulting variable hard sphere gas than in the univariate Maxwellian case corresponding to $\kappa_2 = 0$. The dynamics are obtained through DSMC-sG scheme with $N = 10^5$ particles and $M = 5$ Galerkin projections of the binary collisions. 

\begin{figure}
\centering
\includegraphics[scale = 0.4]{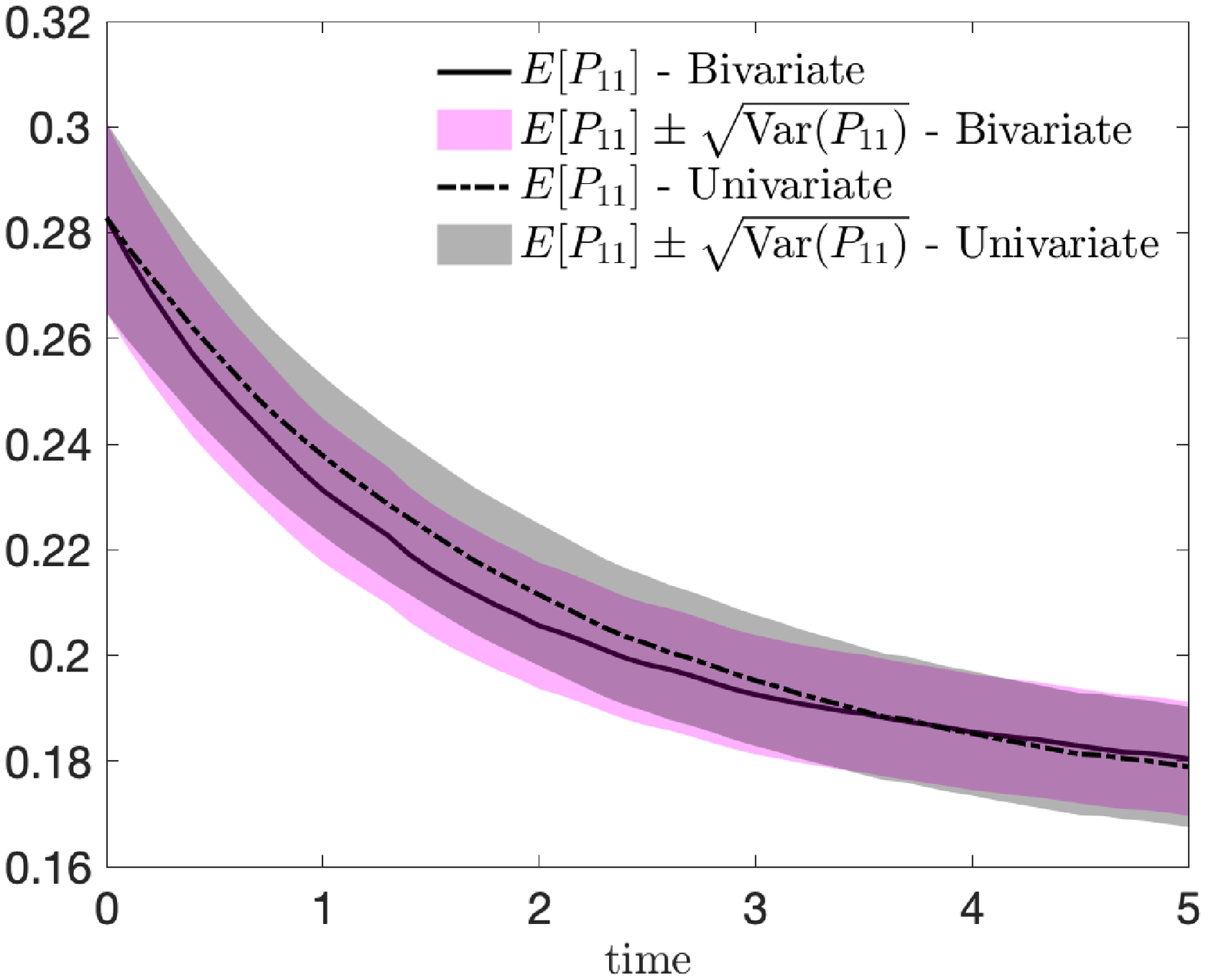}
\includegraphics[scale = 0.4]{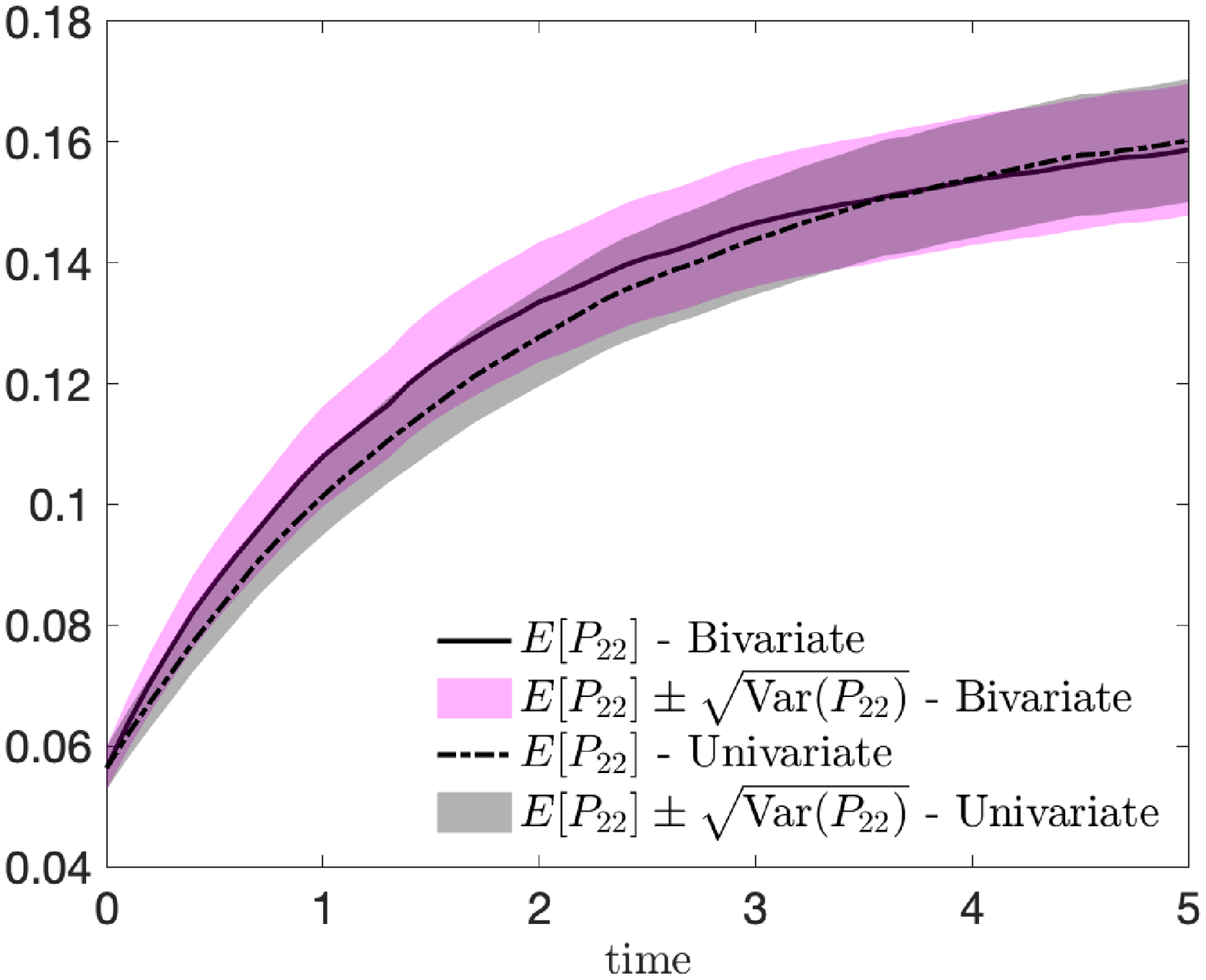}
\caption{ \textbf{Test 4}.  Evolution of the expected stress tensors $P_{11}(z_1,z_2,t)$ and  $P_{22}(z_1,z_2,t)$ from the initial density \eqref{eq:f0_VHS} in the case $\sigma(z_1) = \frac{\lambda \pi}{6}(1+\kappa_1 z_1)$, and uncertain interaction kernel of the form  \eqref{eq:BVHS} characterized by $\gamma(z_2) = \kappa_2(1 + z_2)$. We compare the univariate Maxwellian case corresponding to the choice $\kappa_2 = 0$ with the bivariate case where $\kappa_2 = 1$. In both tests $\kappa_1 = 0.5$ and $z_1,z_2\sim \mathcal U([-1,1])$ are two independent random variables.   }
\label{fig:VHS2D2D}
\end{figure}

\section{Conclusions}
We introduced a novel hybrid approach for uncertainty quantification in collisional kinetic equations of Boltzmann type. The method combines an efficient DSMC solver in the physical space with a stochastic Galerkin method in the random space. This coupling, however, is non trivial since it amounts on a generalized Polynomial Chaos expansion of the statistical samples and on the Galerkin projection of the corresponding DSMC solver. In particular, in the variable hard sphere case, this requires a suitable reformulation of the classical DSMC method. Several numerical examples for space homogeneous problems confirm the efficiency and the accuracy of the resulting solver. We emphasize that the methodology here presented is fully general and can be extended to other Boltzmann-type equations outside the classical rarefied gas dynamics setting. Extension of DSMC-sG methods to semiconductor Boltzmann equations, to the Landau-Fokker-Planck equation of plasma physics and applications to space non homogenous situations are under study and will be presented in forthcoming researches.

\section*{Acknowledgements} L.P. would like to thank the Italian Ministry of Instruction, University and Research (MIUR) to support this research with PRIN Project 2017, No. 2017KKJP4X, "Innovative numerical methods for evolutionary partial differential equations and applications".

\appendix

\section{Exact solutions of Maxwell-type kinetic equations with uncertainty}
In the following we sketch the general methodology to recover the well-known exact solutions of Maxwell-type models\cite{Boby,Ernst,KW} in the presence of uncertain parameters.


\subsection{Kac equation with uncertainty}\label{app:kac}
In Section \ref{sect:kac} we considered the Kac model with uncertain initial temperature.  Let us compute the first moments of the initial density $f_0$ given in \eqref{eq:f0_kac}
such that 
\begin{itemize}
\item[$i)$] $\int_{\mathbb R} f_0(z,v)dv =  \dfrac{\sqrt{\pi}}{2}$ 
\item[$ii)$] $\int_{\mathbb R} v f_0(z,v)dv = 0$
\item[$iii)$] $\int_{\mathbb R}v^2 f_0(z,v)dv = \dfrac{3\sqrt{\pi}}{4\alpha(z)}$.
\end{itemize}
It is easily seen that the total mass of the system is conserved in time whereas the momentum is not conserved for a general initial distribution, indeed it decays to zero with a rate depending on the distribution of the uncertainty, see \cite{TZ}. Thanks to the choice \eqref{eq:f0_kac} we may verify that the initial momentum is conserved in time. 

In order to find explicit solution of the Kac model with uncertainty we can argue as in \cite{Boby,Ernst,KW} and proceed as follows: we consider the class of solutions given by 
\begin{equation}\label{app:general_kac}
f(z,v,t) = (A(z,t)+B(z,t)v^2) e^{- s(z,t) v^2},
\end{equation}
where $A(z,t)$, $B(z,t)$ satisfy the following system by imposing conservation of mass and energy
\[
\begin{cases}\vspace{0.5em}
\sqrt{\pi}\alpha^{3/2}(z)s^{-1/2}(z,t) A(z,t) + \dfrac{\sqrt{\pi}}{2} \alpha^{3/2}(z)  s^{-3/2}(z,t) B(z,t) &= \dfrac{\pi}{2} \\
 \dfrac{\sqrt{\pi}}{2} \alpha^{3/2}(z) s^{-3/2}(z,t)A(z,t) + \dfrac{3\sqrt{\pi}}{4} s^{-5/2}(z,t)B(z,t) &= \dfrac{3\sqrt{\pi}}{4\alpha(z)}, 
\end{cases}
\]
whose solution is given by 
\begin{equation}
\label{app:system}
\begin{cases}
A(z,t) &= \dfrac{3}{4}\sqrt{s(z,t)} \left( \alpha^{-3/2}(z)-\alpha^{-5/2}(z) s(z,t) \right) \\
B(z,t) &= \dfrac{s^{3/2}(z,t)}{2} \left( 3 \alpha^{-5/2}(z)s(z,t) - \alpha^{-3/2}(z) \right). 
\end{cases}
\end{equation}
Therefore, it is sufficient to find exact evolution for $s(z,t)$ to describe the solution of the Kac model. To this end, we substitute the general solution \eqref{app:general_kac} in the collisional equation \eqref{eq:kac}. We have
\[
\begin{split}
\partial_t f(z,v,t) &= \dot{A} e^{-sv^2}+ (-  A\dot{s} + \dot B) v^2 e^{-sv^2}  - B \dot{s} v^4 e^{-sv^2}\\
Q(f,f)(z,v,t) &=  \int_{\mathbb R}\Big[ \dfrac{e^{-s(v^2+v_*^2)}}{2\pi} \int_0^{2\pi} (A+B(v\cos\omega -v_*\sin\omega)^2)(A+B(v\sin\omega +v_*\cos\omega)^2) d\omega \\
& \qquad -(A+Bv^2)(A+Bv^2)e^{-s(v^2+v_*^2)}\Big]dv_*,
\end{split}
\]
where for brevity we omitted explicit dependence on $z$ and time. We may explicitly compute the integrals of the collision operator $Q(f,f)$ to obtain
\[
Q(f,f)(z,v,t) =  \dfrac{3\sqrt{\pi}B^2}{32 s^2 \sqrt{s}} e^{-s v^2}+  \left( \dfrac{\sqrt{\pi}AB}{\sqrt{s}}  - \dfrac{3\sqrt{\pi}}{8}\dfrac{B^2}{s\sqrt{s}} \right)v^2 e^{-sv^2} + \dfrac{\sqrt{\pi}B^2}{8\sqrt{s}} v^4 e^{-sv^2}.
\]
Equating now the terms in $v^{2n}$, $n = 0,1,2$ we reduce to solve
\[
\begin{cases}
\dot A = \dfrac{3\sqrt{\pi}B^2}{32 s^2 \sqrt{s}} \\
-A\dot s + \dot B = \dfrac{\sqrt{\pi}AB}{\sqrt{s}}  - \dfrac{3\sqrt{\pi}}{8}\dfrac{B^2}{s\sqrt{s}} \\
B \dot s = - \dfrac{\sqrt{\pi}B^2}{8\sqrt{s}}.
\end{cases}
\]
By exploiting the relations established in \eqref{app:system} we may observe that the derived conditions are all equivalent to the following differential equation depending on the uncertain parameter $z$ that determines the evolution of $s(z,t)$. We have
\[
\begin{split}
\dot s(z,t) &= - \dfrac{\sqrt{\pi}}{16\alpha^2(z)\sqrt{\alpha(z)}} (3s^2(z,t)-\alpha(z)s(z,t)) \\
s(z,0) &= \alpha(z),
\end{split}\]
whose solution is
\begin{equation}\label{app:s}
s(z,t) = \dfrac{\alpha(z)e^{\frac{1}{16\alpha(z)}\sqrt{\frac{\pi}{\alpha(z)}}t}}{-2 + 3e^{\frac{1}{16\alpha(z)}\sqrt{\frac{\pi}{\alpha(z)}}t}}.
\end{equation}
Therefore, plugging \eqref{app:s} in the definition of $A(z,t)$, $B(z,t)$ in \eqref{app:system} we get \eqref{app:general_kac}, which is the exact solution of the Kac model depending on the uncertain quantity $z$.

\subsection{2D Maxwell models with uncertainty} \label{app:2D}
In Section \ref{sect:2DMax} we considered the 2D Boltzmann model for Maxwell molecules, i.e. $B \equiv 1$. In order to argue as in \cite{Boby,Ernst,KW} we can proceed as in Section \ref{app:kac}: let us consider the following initial distribution with uncertain temperature
\[
f_0(z,v) = \dfrac{\alpha^2(z) \mathbf v^2}{\pi} e^{-\alpha(z)\mathbf v^2},\qquad \mathbf v = \sqrt{v_x^2 + v_y^2}. 
\]
We compute the first moments of $f_0$ to obtain
\begin{itemize}
\item[$i)$] $\int_{\mathbb R^2} f_0(z,v)dv_x\,dv_y = 1$
\item[$ii)$] $\int_{\mathbb R^2} \mathbf v f_0(z,v)dv_x\,dv_y = 0$
\item[$iii)$]$\dfrac{1}{2}\int_{\mathbb R^2} \mathbf v^2 f_0(z,v)dv_x\,dv_y = \dfrac{1}{\alpha(z)} $. 
\end{itemize}
Hence, we consider the class of solutions given by 
\[
f(z,v,t) = (A(z,t) + B(z,t) \mathbf v^2) e^{-\mathbf v^2/2s(z,t)},
\]
and we impose the conservation of mass and energy to obtain
\[
\begin{cases}
2\pi s(z,t) (A(z,t) + 2B(z,t)s(z,t)) = 1, \\
2\alpha(z)\pi s^2(z,t) (A(z,t) + 4B(z,t)s(z,t)) = 1. 
\end{cases}
\]
Analogous computations as in Section \ref{app:kac} yield
\[
f(z,v,t) = \dfrac{1}{2\pi s(z,t)} \left[1 - \dfrac{1-s(z,t)}{s(z,t)}\left( 1-\dfrac{\mathbf v^2}{2s(z,t)} \right) \right]e^{-\mathbf v^2/2s(z,t)},
\]
where the evolution of $s(z,t)$ is given by 
\[
 s(z,t) = \dfrac{1}{\alpha(z)} \left( 1-\dfrac{1}{2}e^{-t/8} \right). 
\]

\section{Stochastic Galerkin representation of random samples and consistency of DSMC-sG approximation}\label{app:particle}
In this appendix we describe how to construct a set of random samples and their gPC projection. Next we derive a consistency estimate for the moments of the resulting empirical density in the DSMC-sG approximation.  

\subsection{Stochastic Galerkin representation of random samples}
Given the initial distribution $f_0(z,v)$ we need to construct a random sample $\tilde v(z) = \{v_i(z)\}_{i=1,\dots,N}$ such that for all $z \in \Omega$ the empirical density 
\[
f^N_0(z,v) = \dfrac{1}{N} \sum_{i=1}^N \delta(v-v_i( z)),
\]
is such that, formally, $f_0^N( z,v) \to  f_0({z},v)$ for $N \to +\infty$. For each $z$ standard random sampling techniques can be employed like direct sampling or acceptance-rejection algorithms, see \cite{Par2} and the references therein for an introduction. A direct application of these methods, however, is not straightforward in the case the samples will be projected using a gPC expansion with respect to the random quantity.

 In the sequel we clarify how we perform projections in the polynomial space $\mathbb P^M$ of a sample. Therefore, we need to determine $v^M_i(z)$ such that 
\[
f_0^N(z,v)\approx \dfrac{1}{N} \sum_{i=1}^N \delta(v-v_i^M(z)), \qquad v_i^M(z)  = \sum_{k=0}^M \hat{v}_{i,k}\Phi_k(z).
\]

The problem can be formulated as follows: given an initial density $f_0(z,v)$ and the distribution of the uncertain quantities $p(z)$, starting from the Gauss collocation nodes $z_0,\dots,z_H$, $H\in\mathbb N$ (chosen following the Wiener-Askey scheme, see \cite{Xiu}), we want to obtain the set of samples
\begin{equation}
\begin{split}
V^0 &= \{ v_1(z_0),\dots,v_N(z_0) \}^T,\\
&\vdots \\
V^H &= \{ v_1(z_H),\dots,v_N(z_H)\}^T.
\end{split}
\label{eq:samp}
\end{equation}
Then, the projection on the $k$-th degree term can be obtained by direct integration
\[
\hat{v}_{i,k} = \int_{I_Z} v_i(z) \Phi_k(z)p(z)dz \approx \sum_{h=0}^H w_h v_i(z_h)\Phi_k(z_h). 
\]
Note that, the above samples \eqref{eq:samp} have to be correlated, namely $v_i(z_h)$, $h=1,\ldots,H$ should represent the same sample $v_i(\cdot)$ for the various values $z_h$. At the numerical level, in 1D it is possible to overcome this issue by considering a set of uncorrelated groups of samples $V^{h}$, $h=1,\ldots,H$ which have been ordered for a given $h$ through a simple sorting process in velocity. This approach, however, cannot be extended directly to higher dimensions in velocity space. In this latter case, there are various possible techniques that can be adopted, accordingly to the particular sampling strategy.

If we consider the multidimensional cases treated in Section \ref{sect:num} where the uncertainties are affecting the temperature of the system defined explicitly in \eqref{eq:T_2D} and \eqref{eq:TE2D2}, an effective strategy is to exploit the classical scaling property of continuous distributions in terms of the second order moment. This is described in  the following algorithm:

\begin{algorithm}[sG projection of initial random sample]~
\begin{enumerate}
\item Generate a set of Gaussian nodes $(z_0,\dots,z_H)$ according to the distribution $p(z)$.
\item \begin{tabbing}
\= {\fP for} \= $i=1$  {\fP to} $N$ \\
\>		\> Generate a sample $v_i(z_0)$ using a suitable sampling method\\ 
\>		\> from the initial density $f_0(z_0,v)$\\
\>		\> Compute $v_i(z_h) = \sqrt{T(z_h)}{v_i(z_0)}$, $h=1,\dots,H$\\  
\> {\fP end for} 
\end{tabbing}
\item \begin{tabbing}
\= {\fP for} \= $m=0$  {\fP to} $M$ \\
	    \>		\> compute projections on $m$-th degree linear space $\{\hat v_{i,m}\}$ \\
         \> {\fP end for} \\
\end{tabbing}
\end{enumerate}
\label{alg:samp0}
\end{algorithm}


\subsection{Consistency estimate}
 The spectral convergence of sG expansion of the samples (see also Section \ref{sect:DSMC}) for sufficiently regular functions in the random space follows form standard results in polynomial approximation theory, we recall for example \cite{Fun, Xiu}.  In particular,  let $H^r(\Omega)$ be a weighted Sobolev space 
\[
H^r(\Omega) = \left\{ v:\Omega\rightarrow \mathbb R: \dfrac{\partial^k v}{\partial z^k} \in L^2(\Omega), 0\le k \le r \right\}. 
\]
Thanks to the introduced properties of the polynomial basis of the random space we have
\begin{lemma}\label{th_fu}
For any $v(z) \in H^r(\Omega)$, $r\ge 0$, there exists a constant $C$ independent of $M>0$  such that
\[
\|v - v_M \|_{L^2(\Omega)} \le \dfrac{C}{M^r} \| v \|_{H^r(\Omega)},
\]
\end{lemma}
Starting from the above spectral estimate, we want to obtain an overall estimate for the moments of the particle distribution computed using the Monte Carlo approach.

Given a function $f(z,v,t)$, we define the expected value of $f$ with respect to $p(z)$ as
\[
\mathbb E[f] = \int_{\Omega} f(z,v,t)p(z)dz,
\]
its empirical measure and the empirical measure in the sG representation as 
\begin{equation}
\label{eq:FN_empi}
f^N(z,v,t) = \dfrac{1}{N} \sum_{j=1}^N \delta(v-v_i(z,t)),\qquad f^N_M(z,v,t) = \dfrac{1}{N} \sum_{j=1}^N \delta(v-v_i^M(z,t)).
\end{equation}
Observe that, for any a test function $\varphi$, if we denote by 
$$\langle \varphi, f \rangle(z,t):=\int_{\RR^{d}}f(z,v,t)\varphi(v)\,dv,$$ 
we have
\[
\langle \varphi, f^N \rangle(z,t) = \dfrac{1}{N} \sum_{j=1}^N \varphi(v_i(z,t)),\qquad \langle \varphi, f_M^N \rangle(z,t) = \dfrac{1}{N} \sum_{j=1}^N \varphi(v^M_i(z,t)). 
\]
Note that, if we assume that $\int_{\R^d} f(z,v,t)\,dz = 1$,  then
 $\langle \varphi, f \rangle(z,t)$ denotes the expectation of $\varphi$ with respect to $f$. Therefore, from the central limit theorem we have \cite{Caflisch}
\begin{lemma}
The root mean square error satisfies
\[
{\mathbb E}_V\left[\left(\langle \varphi, f \rangle(z,t)-\langle \varphi, f^N \rangle(z,t)\right)^2 \right]^{1/2} = \frac{\sigma_\varphi(z,t)}{{N}^{1/2}}
\]
where
\[
\sigma^2_\varphi(z,t) = \int_{\R^d} \left(\varphi(v)-\langle \varphi, f \rangle(z,t)\right)^2\, f(z,v,t)\,dv.
\]
\label{le:2}
\end{lemma}
In the above lemma we used the notation ${\mathbb E}_V$ to denote the expectation in the velocity space with respect to $f$. More precisely, for each $z\in\Omega$, $\langle \varphi, f^N \rangle(z,t)$ is considered as the sum of $N$ random variables $\varphi(v_1(z,t)),\ldots,\varphi(v_N(z,t))$ with $v_1(z,t),\ldots,v_N(z,t)$ independent and identically distributed as $f(z,v,t)$.

Next, for a random variable $V(z,t)$ taking values in $L^2(\Omega)$ we define 
\[
\| V \|_{L^2(\Omega;L^2(\mathbb R^{d_v}))} = \| \mathbb E_V[V^2]^{1/2} \|_{L^2(\Omega)},
\]
or equivalentely  
\[
\| V \|_{L^2(\mathbb R^{d_v};L^2(\Omega))} =\mathbb E_V\left[\|V\|^2_{L^2(\Omega)}\right]^{1/2}.
\]

We have the following result:
\begin{theorem}
Let $f(z,v,t)$ a probability density function in $v$ at time $t\ge 0$ and $f_M^N(z,v,t)$ the empirical measure of the $N$-particles sG approximation with $M$ projections associated to the samples $\{v_1(z,t),\dots,v_N(z,t)\}$ defined in \eqref{eq:FN_empi}. Provided that $v_i(z,t)\in H^r(\Omega)$ for all $i=1,\dots,N$, the following estimate holds
\begin{equation}
\| \langle\varphi,f\rangle - \langle \varphi, f_M^N\rangle \|_{L^2(\mathbb R^{d_v};L^2(\Omega))} \le \dfrac{\|\sigma_{\varphi} \|_{L^2(\Omega)}}{{N}^{1/2}}+\dfrac{C}{M^r} \left( \dfrac{1}{N} \sum_{i=1}^N \| \nabla\varphi(\xi_i)\|_{L^2(\Omega)}\right),
\end{equation}
where $\varphi$ is a test function, $C$ is a positive constant independent on $M$, $\xi_i  = (1-\theta)v_i + \theta v_i^M$, $\theta \in (0,1)$. 
\end{theorem}
\begin{proof}

Thanks to the properties of the  norm we have
\begin{equation*}\begin{split}
&\| \langle\varphi,f\rangle - \langle\varphi,f^N_M\rangle  \|_{L^2(\mathbb R^{d_v};L^2(\Omega))}  \\
&\qquad \le
\underbrace{\|  \langle\varphi,f\rangle -  \langle\varphi,f^N\rangle \|_{L^2(\mathbb R^{d_v};L^2(\Omega))}}_{I} + \underbrace{ \|   \langle\varphi,f^N\rangle - \langle\varphi,f^N_M\rangle  \|_{L^2(\mathbb R^{d_v};L^2(\Omega))}}_{II}.
\end{split}\end{equation*}
The first term can be easily estimated using Lemma \ref{le:2} to get
\[\begin{split}
I 
& = \dfrac{\| \sigma_\varphi(z)\|_{L^2(\Omega)}}{{N}^{1/2}}. 
\end{split}\]
Let us consider now the second term
\[
\begin{split}
II  &= \left\| \dfrac{1}{N} \sum_{i=1}^N (\varphi(v_i)-\varphi(v_i^M)) \right\|_{L^2(\mathbb R^{d_v};L^2(\Omega))} \\
&\le \dfrac{1}{N} \sum_{i=1}^N  \| \varphi(v_i) - \varphi(v_{i}^M) \|_{L^2(\mathbb R^{d_v};L^2(\Omega))}.
\end{split}\]
From the mean value theorem $ \varphi(v_i) - \varphi(v_{i}^M) = \nabla \varphi(\xi_i) \cdot (v_i-v_i^M)$ for $\xi_i  = (1-\theta)v_i + \theta v_i^M$, $\theta \in (0,1)$. Therefore we have
\[
II \le \dfrac{1}{N} \sum_{i=1}^N  \| \nabla\varphi(\xi_i)\|_{L^2(\Omega)} \| v_i-v_i^M\|_{L^2(\Omega)},
\]
and using Lemma \ref{th_fu}, for $C = \max_i C_i \| v_i\|_{H^r(\Omega)}$, we obtain
\[
II \le \dfrac{C}{M^r} \left(\dfrac{1}{N} \sum_{i=1}^N  \| \nabla\varphi(\xi_i)\|_{L^2(\Omega)}\right). 
\]

\end{proof}

\end{document}